\documentclass{article}
\usepackage{amsmath, amsthm, amssymb, graphicx,caption,subcaption,color,colortbl}
\makeindex

\newtheorem{theorem}{Theorem}
\newtheorem*{theorem*}{Theorem}
\newtheorem{corollary}[theorem]{Corollary}
\newtheorem{proposition}[theorem]{Proposition}
\newtheorem{lemma}[theorem]{Lemma}
\newtheorem{conjecture}{Conjecture}
\newtheorem*{conjecture*}{Conjecture}

\newtheorem*{wilson}{Wilson's Theorem}
\newtheorem*{dirichlet}{Dirichlet's Theorem}
\newtheorem*{chev}{Chevalley-Warning Theorem}

\theoremstyle{definition}
\newtheorem*{example}{Example}
\newtheorem{definition}[theorem]{Definition}
\newtheorem{remark}[theorem]{Remark}

\makeatletter
\newcommand\xleftrightarrow[2][]{%
  \ext@arrow 9999{\longleftrightarrowfill@}{#1}{#2}} 
\newcommand\longleftrightarrowfill@{%
  \arrowfill@\leftarrow\relbar\rightarrow}
\makeatother

\def\V{\mathcal{V}}
\def\E{\mathcal{E}}
\def\lcm{\text{lcm}}

\graphicspath{{./Figures/}}
\allowdisplaybreaks[1]

\begin{document}

\title{Properties of the Extended Graph Permanent}
\author{Iain Crump\\Simon Fraser University, Burnaby, B.C.}
\maketitle

\normalsize
\vspace{1cm}

\begin{abstract}

We create for all graphs a new invariant, an infinite sequence of residues from prime order finite fields, constructed from the permanent of a reduced incidence matrix. Motivated by a desire to better understand the Feynman period in $\phi^4$ theory, we show that this invariant is preserved by all graph operations known to preserve the period. We further establish properties of this sequence, including computation techniques and alternate interpretations as the point count of a novel polynomial. 
\end{abstract}

\section{Introduction}

Let $G$ be a graph with $n$ edges. The \emph{period} of $G$ is a residue of the Feynman integral of $G$ viewed as a Feynman diagram in massless scalar field theory. This paper takes its physical motivation from graphs in $\phi^4$ theory, a field theory in which vertices have degree at most four. In particular, a $k$-point $\phi^4$ graph $G$ has $\sum_{v \in V(G)} \deg (v) = 4|V(G)|-k$, and when considering the motivating physics we restrict to $4$-point graphs in $\phi^4$ theory. These graphs can be uniquely derived from $4$-regular graphs by deleting a single vertex. Assigning variables $x_e$ to all edges $e \in E(G)$, the Feynman period in $\phi^4$ theory is $$\int_{x_2 \geq 0} \cdots \int_{x_n \geq 0} \frac{1}{\Psi^2|_{x_1=1}} \prod_{j=2}^{n}\mathrm{d}x_j,$$ where $\Psi$ is the Kirchhoff polynomial, \begin{align} \label{kirchhoff}\Psi = \sum_{\substack{T\text{ spanning} \\ \text{tree of }G}}\prod_{e\not\in E(T)}x_e.\end{align} This is a simplified version of the full Feynman integral, which would normally contain terms in the numerator related to the masses, momenta, and general structure of the graph. The period nonetheless retains much of the number theoretic content of the full Feynman integral.

The Feynman period plays an important role physically. The Feynman integral typically diverges, and the period appears as the coefficient at infinity. Further, it is known precisely when the period diverges for $4$-point $\phi^4$ graphs. The \emph{loop number} or \emph{first Betti number} of a graph $G$, $h_G$, is the minimum number of edges that must be removed to produce an acyclic graph. A graph is \emph{primitive} if $|E(G)| = 2h_G$ and for all non-empty proper subgraphs $H$, $|E(H)| > 2h_H$. It is known (Proposition 5.2 in \cite{BlEsKr}) that the period of a $\phi^4$ graph converges if and only if the graph is primitive. When primitive, the period is invariant under choice of renormalization scheme.

While computationally difficult, there are three graphic operations known to preserve the period in $\phi^4$ theory; the Schnetz twist, completion followed by decompletion, and planar duality (see \cite{Sphi4}). All $4$-point graphs in scalar $\phi^4$ theory can be derived from a $4$-regular graph by deleting a single vertex. This deletion is known as \emph{decompletion}, and the unique way of adding a vertex back to create a $4$-regular graph is \emph{completion}. The Schnetz twist is an operation on the edges incident to a $4$-vertex cut on a completed graph, and can be seen in Figure \ref{twist}. There is an additional computational convenience in that $4$-point $\phi^4$ graphs with $2$-vertex cuts can be split across the cut as in Figure \ref{2cut}, and the period of the original graph is the product of the periods of the two minors. Currently, all known instances where $4$-point $\phi^4$ graphs have equal periods can be explained by these operations.

It follows, then, that graph invariants that are preserved by these operations are of particular interest, as they may further our understanding of the period. Currently, the \emph{$c_2$ invariant} (see \cite{BrS}) and the \emph{Hepp bound} (see \cite{Schhepp}) are conjectured to do this. The \emph{graph permanent}, introduced in \cite{crump}, is a non-trivial invariant that is known to be preserved by all these operations. Derived from the permanent of a matrix created from a signed incidence matrix and an arbitrary orientation of the graph, the graph permanent is unfortunately only defined for graphs $G$ such that $|E(G)| = k(|V(G)|-1)$ for some $k \in \mathbb{Z}_{>0}$ (which does include $4$-point $\phi^4$ graphs, at $k=2$), and produces a value in $\{ 0, \pm 1, \pm 2, ... , \pm \lfloor \frac{k+1}{2} \rfloor \}$ for these graphs. For $4$-point $\phi^4$ graphs, it therefore is a binary invariant, producing values in $\{0, \pm 1\}$.

Herein, we produce a natural extension of the graph permanent, the \emph{extended graph permanent}. It is an extension in that it produces an infinite, nontrivial sequence for a graph, hence potentially providing more information for each graph, and further that it is defined for all graphs. In Section \ref{graphoperations}, we will show that the extended graph permanent behaves as desired under the aforementioned graph operations.

\begin{theorem} Let $\Gamma$ be a $4$-regular graph. 
\begin{itemize}
\item Any two decompletions of $\Gamma$ have the same extended graph permanent. (Theorem \ref{egpcompletion})
\item If $\Gamma$ and $\Gamma'$ differ by a Schnetz twist, then any decompletions of $\Gamma$ and $\Gamma'$ have equal extended graph permanents. (Proposition \ref{schnetz})
\item If a graph $G = \Gamma -v$, $v \in V(\Gamma)$, is planar and has planar dual $G^*$, the extended graph permanents of $G$ and $G^*$ are equal. (Theorem \ref{dual})
\end{itemize} \end{theorem}

\noindent This gives rise to the following conjecture.

\begin{conjecture*}[Conjecture \ref{obviousconjecture}] If two $4$-point $\phi^4$ graphs have equal periods, then they have equal extended graph permanents. \end{conjecture*}

\noindent The converse of this conjecture appears to be false; multiple sets of graphs appear to have equal sequences, which can be seen in Appendix \ref{chartofgraphs}. These theorems and this conjecture do suggest, though, that the extended graph permanent may be a useful invariant in the study of the Feynman period, as it may create additional methods to understanding the Feynman period.

The extended graph permanent is constructed from a matrix permanent, so cofactor expansion is a useful method of computation. Cofactor expansion can actually be used to produce a closed form for the values of the sequence. This will be examined in Section \ref{egpcomp}, and closed forms for the zig-zag and wheel families are developed. 

Some of the first sequences examined were familiar, comparable up to sign with $c_2$ invariant sequences from \cite{BSModForms}. The values in $c_2$ sequences correspond to residues in $\mathbb{F}_p$ for all primes $p$. In turn, it was observed that these $c_2$ values match the $p^\text{th}$ Fourier coefficient of modular forms over the integers taken modulo $p$ for a finite number of initial primes for which they were computed. A number of these sequences were further shown to match the modular forms completely in \cite{BSModForms} and \cite{Logan}. As such, it was asked if our sequences could be constructed like the $c_2$ invariant; as the point count of a polynomial over finite fields\footnote{Special thanks to Dr. Francis Brown, who first posed this question, and also directed us to the Chevalley-Warning Theorem and Theorem \ref{chevwarn}.}. We explore this in section \ref{affinev}, and in doing so derived the following novel polynomial. Let $G$ be a graph with an arbitrary edge orientation. For $v \in V(G)$, let $\delta^+(v)$ be the set of edges in an oriented graph directed towards a vertex $v$ and $\delta^-(v)$ the edges oriented away from $v$. Define $L = \text{lcm}(|E(G)|,|V(G)|-1)$, $\V = \frac{L}{|V(G)|-1}$, $\E = \frac{L}{|E(G)|}$, and $G^{[\E]}$ the graph created from $G$ by replacing all edges with $\E$ edges in parallel, preserving orientation. Create a variable $x_e$ for all edges $e \in E \left( G^{[\E]} \right)$ and choose an arbitrary vertex $v' \in V(G)$. We then define the polynomial $$\widetilde{F}_{G,v'} = \prod_{\substack{v \in V(G) \\ v \neq v'}} \left( \sum_{e \in \delta^+(v)} x_e^\V - \sum_{e \in \delta^-(v)} x_e^\V \right).$$  While this polynomial is dependent on the choice of vertex $v'$, our interest lies in the point count over a finite field $\mathbb{F}_p$; the number of zeroes of the polynomial over this finite field. For a polynomial $f$, we denote the point count over $\mathbb{F}_p$ as $[f]_p$. Let $\text{GPerm}^{[p]}(G)$ be the value of the extended graph permanent for a graph $G$ at prime $p$.

\begin{theorem}[Corollary \ref{maincor}] Let $G$ be a (not necessarily connected) $4$-point $\phi^4$ graph, and $v' \in V(G)$. Regardless of choice of vertex $v'$ and for odd prime $p$, the extended graph permanent of $G$ at $p$ is $$ \text{GPerm}^{[p]}(G) \equiv \begin{cases} [\widetilde{F}_{G,v'}]_p \pmod{p} \text{ if } |E(G)| \equiv 0 \pmod{4} \\ -[\widetilde{F}_{G,v'}]_p \pmod{p} \text{ otherwise} \end{cases} .$$ \end{theorem}

A number of extended graph permanent sequences also appear to relate to modular forms, as the $p^\text{th}$ Fourier coefficient modulo $p$. This is discussed in greater detail in Section \ref{modformscoeffs}. It is interesting to note that in all observed instances, the loop number of the graph is equal to the weight of the modular form, and the level of the modular form is a power of two.

\section{The Extended Graph Permanent}\label{Extended}

In the following subsection we introduce the matrix permanent and some notational conventions that we will follow, and further establish some properties that are necessary to create our invariant.  Much of this was previously introduced and discussed in \cite{crump}. The subsection that follows introduces the extended graph permanent. 

\subsection{Properties of the matrix permanent} \label{propsoftheperm}

 For notational convenience we will use the \emph{Kronecker product} to construct block matrices. For matrices $A=(a_{i,j})$ and $B$, $$A \otimes B = \left[ \begin{array}{ccc} a_{1,1}B & a_{1,2} B & \cdots \\ a_{2,1} B & a_{2,2} B & \cdots \\ \vdots && \end{array} \right].$$ We will denote the $n \times m$ matrix with all entries $t$ by $\mathbf{t}_{n \times m}$, or simply $\mathbf{t}_n$ if it is an $n \times n$ square. We denote the $n \times n$ identity matrix as $I_n$, or $I$ if the dimension is clear from context.
 
\begin{definition} Let $M$ be a matrix. We define the \emph{fundamental block matrix of $M$}, $\overline{M}$, to be the smallest square matrix that can be created using blocks of $M$. That is, the fundamental matrix is the smallest square matrix of the form $\mathbf{1}_{m \times n} \otimes M$ for positive integers $m$ and $n$. \end{definition}

For a graph $G$, we let $M_G^* = [m_{i,j}]$ be a signed incidence matrix of $G$, where rows are indexed by vertices and columns are indexed by edges; $$ m_{v,e} =  \begin{cases} 1, & \text{ if } h(e)=v  \\ -1, & \text{ if } t(e)=v \\ 0, & \text{ otherwise}  \end{cases} ,$$ where $h(e)$ is the head of edge $e$ and $t(e)$ is the tail.  

\begin{definition} Let $G$ be a connected graph. Arbitrarily apply directions to the edges in $G$, and  let $M^*_G$ be a signed incidence matrix associated with this digraph. Select a vertex $v$ in $V(G)$, and delete the row indexed by $v$ in $M^*_G$. Call this new matrix $M_G$. Let $\overline{M}_G$ be the fundamental matrix of $M_G$ (or \emph{a fundamental block matrix of $G$}, dependent on the orientation and choice of special vertex). Let $L = \text{lcm}\{|V(G)|-1, |E(G)|\}$. We then regularly use \begin{align} \E &= \frac{L}{|E(G)|} \label{edef} \\ \V &= \frac{L}{|V(G)|-1} \label{vdef},  \end{align} as $\overline{M}_G = \mathbf{1}_{\V \times \E} \otimes M$. Further, we call $v$ the \emph{special vertex} in the construction of $M$. As noted prior, graphs with $|E(G)| = k(|V(G)|-1)$ for some $k \in \mathbb{N}$ are of particular interest, so we define the \emph{$k$-matrix} of a matrix $M$ to be the block matrix $\mathbf{1}_{k \times 1} \otimes M$. \end{definition}

\begin{example} Consider the complete graph $K_3$, shown below. We select the marked vertex as the special vertex and orient as indicated. This results in the fundamental matrix $\overline{M}_G$.
$$G = \raisebox{-.48\height}{\includegraphics{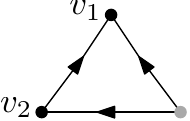}} \hspace{5mm} \overline{M}_G= \begin{array}{r} v_1\\v_2\\ \phantom{v_1} \\ \phantom{v_1} \\ \phantom{v_1}\\ \phantom{v_1} \end{array} \left[ \begin{array}{ccc|ccc} 1 & 0 & 1 & 1 & 0 & 1  \\ -1 & 1 & 0 & -1 & 1 & 0 \\ \hline 1 & 0 & 1 & 1 & 0 & 1  \\ -1 & 1 & 0 & -1 & 1 & 0 \\ \hline 1 & 0 & 1 & 1 & 0 & 1  \\ -1 & 1 & 0 & -1 & 1 & 0  \end{array} \right]$$

Similarly, for $K_4$, we produce the following fundamental matrix.

$$G = \raisebox{-.48\height}{\includegraphics[scale=.9]{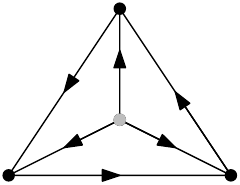}} \hspace{5mm} \overline{M}_G= \left[ \begin{array}{cccccc} 
1 & 0 & 0 & 1 & -1 & 0  \\ 
0 & 1 & 0 & -1 & 0 & 1 \\ 
0 & 0 & 1 & 0 & 1 & -1 \\ \hline
1 & 0 & 0 & 1 & -1 & 0  \\ 
0 & 1 & 0 & -1 & 0 & 1 \\ 
0 & 0 & 1 & 0 & 1 & -1  \end{array} \right]$$

\end{example}

Decompleted $4$-regular graphs (like $K_4$) will always have fundamental matrices that are a $2$-matrix, as for all graphs $G$ in this family, $|E(G)| = 2(|V(G)|-1)$.

We desire a square matrix to allow for standard matrix operations. While this construction will result in zero matrix determinants due to the duplicated rows for all non-tree graphs, the permanent is still of interest. 

\begin{definition}\label{permdef} Let $A=(a_{i,j})$ be an $n$-by-$n$ matrix. The \emph{permanent} of $A$ is $$\text{Perm}(A) = \sum_{\sigma \in S_n} \prod_{i=1}^n a_{i,\sigma(i)},$$ where the sum is over all elements of the symmetric group $S_n$. \end{definition}

This is the Leibniz formula for the permanent, which differs from the determinant in that there are no signs. Unfortunately, the permanent is not invariant under adding a multiple of a row to another row. As a result, and in an effort to regain row reduction techniques, we focus instead on a residue of the permanent.

\begin{proposition}[Corollary 6 in \cite{crump}] \label{reduction} Suppose $M = \mathbf{1}_{k \times 1} \otimes B$ is square for some matrix $B$, and $r_i$ and $r_j$ are rows of $M$ in a common block, $i \neq j$. Let $M'$ be a matrix derived from $M$ by adding a constant integer multiple of $r_j$ to $r_i$ in each block. Then $\text{Perm}(M) \equiv \text{Perm}(M') \pmod {k+1}$. \end{proposition}

It is important to note that the fundamental block matrix does indeed have this form, as $$\mathbf{1}_{m \times n} \otimes M = \mathbf{1}_{m \times 1} \otimes \left( \mathbf{1}_{1 \times n} \otimes M  \right).$$

\begin{remark} \label{rowops} Trivially, the permanent is preserved under interchanging rows or columns, and behaves like the determinant when multiplying a row or column by a constant. With Proposition \ref{reduction}, the residue  $\text{Perm}(\mathbf{1}_{k \times 1} \otimes M) \pmod{k+1}$ is well-behaved under row operations within all blocks simultaneously given this matrix construction. \end{remark}

\begin{theorem}[Theorem 9 in \cite{crump}]\label{specialinvariant} Let $G$ be a graph and fix an orientation to the edges. Let $\overline{M} = \mathbf{1}_{n \times m} \otimes M$ be the fundamental matrix from this orientation. The permanent of this matrix is invariant under choice of special vertex modulo $n+1$. \end{theorem}

Theorem \ref{specialinvariant} immediately gives us the option to change which vertex is the special vertex in a calculation. This will be an important tool later.

\begin{remark}\label{kfactorial} (Lemma 4 in \cite{crump}) If a matrix has $k$ identical rows or columns, then $k!$ divides the permanent. \end{remark}

\begin{theorem}[Proposition 13 in \cite{crump}]\label{poof} For non-prime $k + 1$, the permanent of any fundamental matrix $\mathbf{1}_{k \times n} \otimes M$ associated to a graph $G$ with $|V(G)| > 2$ is zero modulo $k + 1$.   \end{theorem}

\begin{proof} As $|V(G)|>2$, $k!^2$ is a factor in the permanent by Remark \ref{kfactorial}. Factoring $k+1 = ab$ where $a,b>1$, both appear in the product $k!$, and the result follows. \end{proof}

\subsection{Constructing the extended graph permanent}

From Theorem \ref{poof}, we have seen that only prime residues are of interest when computing permanents. The following classical theorem, coupled with Theorem \ref{poof}, is key to our construction of sequences based on the permanent.

\begin{dirichlet} For relatively prime $a$ and $b$, the sequence $(an+b)_{n \in \mathbb{N}}$ contains infinitely many primes. \end{dirichlet}

It follows that there are infinitely many primes of the form $an+1$ for arbitrary positive integer $a$.

\begin{definition} Let $G$ be a graph and $\overline{M} = \mathbf{1}_{\V \times \E} \otimes M$ a fundamental matrix of $G$. Define $(p_i)_{i \in \mathbb{N}} $ as the increasing sequence of all primes that can be written $p_i = n_i\V+1$ for some positive integer $n_i$. Then, matrix $\mathbf{1}_{n_i} \otimes \overline{M}$ is square and each row appears $n_i\V$ times. As such, the permanent is well-defined modulo $n_i\V+1 =p_i$. Call this residue the \emph{$p_i^\text{th}$ graph permanent}, $\text{GPerm}^{[p_i]}(G)$. Define the \emph{extended graph permanent} as the sequence $$\left( \text{GPerm}^{[p_i]}(G) \right)_{i \in \mathbb{N}}.$$ \end{definition}

Trees, which uniquely produce sequences that have values at all primes, will be discussed in Section \ref{treesnshit}. The $4$-point graphs in $\phi^4$ theory, our motivating class, produce sequences with values at all odd primes. 

The extended graph permanent relies on the arbitrary orientation of edges in a graph in the construction of the matrix. As changing the orientation is equivalent to multiplying a column of the signed incidence matrix by $-1$, there is potentially a sign ambiguity associated to this permanent. However, as the definition fixes an orientation for all copies of the edge-defined columns, this sign ambiguity occurs only over primes that require an odd number of duplications of columns, and the ambiguity affects all values of this type together. The sign ambiguity will be discussed in greater detail in Section \ref{signambiguity}.

\begin{remark} \label{connected} While the definition makes no mention of connectedness of the graph, a connected component that does not contain the special vertex will cause the permanent to vanish for all primes. This is consistent with the quantum field theory motivation.  

Interestingly, if we instead require one special vertex per connected component, the matrix again becomes full rank. It is impossible in this matrix to differentiate between this disconnected graph and a similar connected graph where the special vertex in each connected component is identified, resulting in a cut vertex. By Theorem \ref{specialinvariant}, we may therefore cleave a graph at a cut vertex, switch which vertex is special in each component, and then identify the special vertices again. \end{remark}

\section{Invariance Under Period Preserving Operations} \label{graphoperations}

We return now to our motivation, the Feynman period, and show that the extended graph permanent is invariant under all graphic operations known to preserve the period. A useful tool will be a way of interpreting the column duplication in the construction of larger matrices as an operation on the graph.

\begin{remark}\label{altrem} We may alternately define the extended graph permanent in a more structural setting. Create the \emph{$n$-duplicated graph} $G^{[n]}$ by replacing all edges in $G$ with $n$ edges in parallel. Let $M_n$ be a signed incidence matrix of $G^{[n]}$ with some choice of special vertex deleted, such that all edges in parallel are oriented in the same direction. Then, when there are values $k,n \in \mathbb{N}$ such that $\mathbf{1}_{k \times 1} \otimes M_n$ is square and $k+1 = p$ is prime, $\text{GPerm}^{[p]}(G) = \text{Perm}(\mathbf{1}_{k \times 1}\otimes M_n)$. \end{remark}

A key to understanding the permanent computation for a $k$-matrix associated to a graph $G$ with $|E(G)| = k(|V(G)|-1)$ is as follows. Each non-zero value in the permanent computation sum is determined by a selection of non-zero entries in the matrix, one entry in each row and one in each column. Call such selections a \emph{contribution}. Associate each of the $k$ blocks with a unique colour. Then, a non-zero value is selected once in each row and once in each column, corresponding to each edge being selected once, and each non-special vertex being selected $k$ times. We may associate this selection then to a colouring and tagging of the edges. Each edge receives a colour corresponding to the block that it was selected in, and a tag at the vertex that selected that edge. The value of each contribution, then, is the product of the entries selected. By construction, if $t$ is the number of edge tags that lie on the tails of arcs in the underlying orientation, the value of the contribution is $(-1)^t$.

\begin{remark}[Remark 11 in \cite{crump}]\label{casualintro} There is a bijection between these taggings and colourings and the contributions to the permanent. \end{remark}

Each tagging allows $k!^{|V(G)|-1}$ colourings, since each non-special vertex receives a tag from precisely one edge of each colour. Hence, we may work with these contributions in a structural sense by considering only the taggings.

\begin{example} Again, consider the graph $K_3$. A contribution to the permanent, and the associated edge tagging, is shown below. Again, the special vertex is labeled in gray.

 \centering     \includegraphics[scale=1.0]{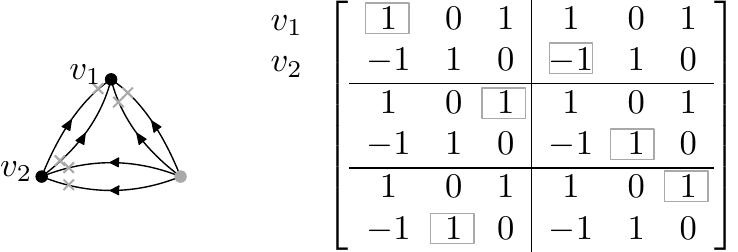}

 \end{example}

%\begin{theorem}\label{egpcompletion} Let $G$ be a $2k$-regular graph. For any choice of $v\in V(G)$, $G - v$ has the same extended graph permanent.  \end{theorem}

%\begin{proof} Let $v,w \in V(G)$. We prove this by showing that for any odd prime $p = nk+1$, there is an orientation of the edges of $G_v = G-v$ and $G_w = G - w$ such that $\text{GPerm}^{[p]}(G_v) = \text{GPerm}^{[p]}(G_w)$. Let $w$ be the special vertex for $G_v$, and similarly let $v$ be the special vertex for $G_w$.

%For a contribution to the permanent of $G_v$ for prime $p$, use the graph $(G_v)^{[n]}$ per Remark \ref{altrem}. Extend such a tagging to the graph $G^{[n]}$ by adding vertex $v$ back with all incident edges and duplications, so that vertex $v$ receives all tags from incident edges. Apply an orientation to the edges so that all edges incident to $v$ are oriented towards $v$, and all edges incident to $w$ are oriented away from $w$. The remaining edges may be oriented arbitrarily.

%We bijectively move between such a tagging of $(G_v)^{[n]}$ and $(G_w)^{[n]}$ by reversing the orientation of all tags, thus reversing the roles of $v$ and $w$ as the special and decompletion vertices. Further, reverse the underlying orientation of all edges. In doing so, the values of the mapping between contributions is fixed for these orientations of $G_v$ and $G_w$, and thus the extended graph permanents are equal.  \end{proof}

\begin{remark} \label{evendegree} Suppose $G$ is a $d$-regular graph, and we are considering a decompletion $G-v$. Then, in computing the extended graph permanent $G-v$ has $|V(G)|-2$ vertices that receive tags, and $\frac{d(|V(G)|-2)}{2}$ edges. As each of these $|V(G)|-2$ vertices must receive an equal number of tags, each receives $d/2$ tags, precisely half the number of incident edges. It follows that such a $d$-regular graph must either have $d$ even, or if $d$ is odd then the number of duplications $n$ in $G^{[n]}$ must be even when computing the extended graph permanent for defined primes. \end{remark}

\begin{theorem}\label{egpcompletion} Let $G$ be a regular graph. For any choice of $v\in V(G)$, $G - v$ has the same extended graph permanent.  \end{theorem}

\begin{proof} Let $v,w \in V(G)$. We prove this by showing that for any odd prime $p$, there is an orientation of the edges of $G_v = G-v$ and $G_w = G - w$ such that $\text{GPerm}^{[p]}(G_v) = \text{GPerm}^{[p]}(G_w)$. Let $w$ be the special vertex for $G_v$, and similarly let $v$ be the special vertex for $G_w$.

For a contribution to the permanent of $G_v$ for prime $p$, use the graph $(G_v)^{[n]}$ per Remark \ref{altrem}. Extend such a tagging to the graph $G^{[n]}$ by adding vertex $v$ back with all incident edges and duplications, so that vertex $v$ receives all tags from incident edges. Apply an orientation to the edges so that all edges incident to $v$ are oriented towards $v$, and all edges incident to $w$ are oriented away from $w$. The remaining edges may be oriented arbitrarily.

We bijectively move between such a tagging of $(G_v)^{[n]}$ and $(G_w)^{[n]}$ by reversing the orientation of all tags, thus reversing the roles of $v$ and $w$ as the special and decompletion vertices. All other vertices will still receive half of the tags from incident edges per Remark~\ref{evendegree}, and hence this is still a valid tagging of the graph. Further, reverse the underlying orientation of all edges. In doing so, the values of the mapping between contributions is fixed for these orientations of $G_v$ and $G_w$, and thus the extended graph permanents are equal.  \end{proof}

The Schnetz twist is another operation known to preserve the period (\cite{Sphi4}). Shown in Figure \ref{twist}, we partition the edges of a graph across a four-vertex cut, and on one side redirect edges incident to vertices of the cut. We assume that both graphs are $4$-regular. If two graphs differ by a Schnetz twist, then any pair of decompletions of the graphs are known to have equal periods.

\begin{figure}[h]
  \centering
      \includegraphics[scale=0.80]{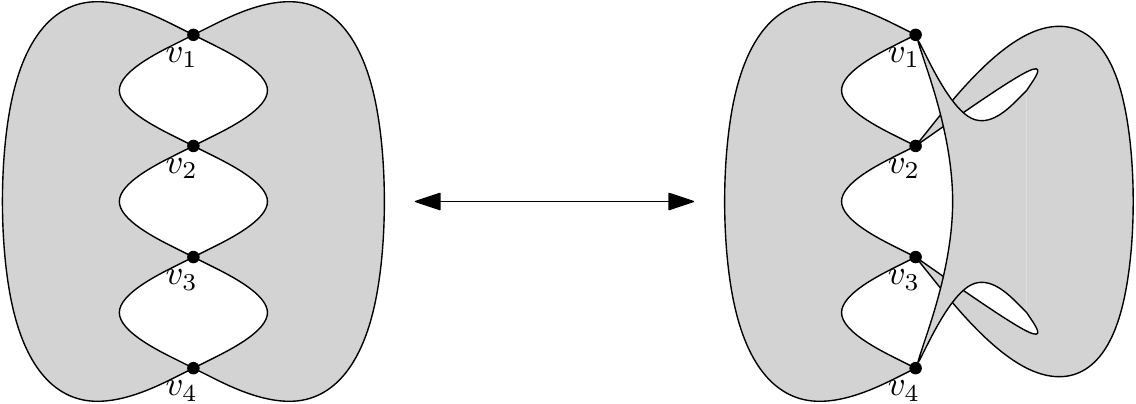}
  \caption{The Schnetz twist. If both graphs are $4$-regular, then all decompletions of these two graphs have equal Feynman periods.}
\label{twist}
\end{figure}

We extend the notion of the Schnetz twist to allow for both graphs to be $d$-regular.

\begin{proposition}\label{schnetz} Consider two $d$-regular graphs that differ by a Schnetz twist, say $G_1$ and $G_2$. Decompletions of these graphs have equal extended graph permanents. \end{proposition}

\begin{proof} Label the vertices in the four-vertex cut as in Figure \ref{twist}.  By Theorem \ref{specialinvariant} and Theorem \ref{egpcompletion}  we may chose $v_3$ as the special vertex and $v_4$ as the decompletion vertex for both graphs. For prime $p$, we again extend the contributions to $G_1^{[n]}$ and $G_2^{[n]}$ by saying that the decompletion vertex receives all tags. These are both $dn$-regular graphs.

Fix a contribution to the permanent in $G_1^{[n]}$. Then, the decompletion vertex receives $dn$ tags, the special vertex receives none, and all others get $\frac{dn}{2}$ tags. Since we assume $G_1^{[n]}$ and $G_2^{[n]}$ are both $dn$-regular for even $dn$ by Remark~\ref{evendegree}, suppose vertex $v_1$ is incident to $d_1$ edges on the left, and $v_3$ is incident to $d_3$ on the left. Then, vertices $v_2$ and $v_4$ must be incident to $d_1$ and $d_3$ edges on the left, respectively. If there are $v$ vertices properly contained on the left, then there are $$\frac{1}{2}(dnv+2d_1+2d_3) = \frac{dnv}{2}+d_1+d_3$$ edges, and hence total tags, on the left. Each of the $v$ vertices properly on this side receive $\frac{dn}{2}$ tags, while $v_3$ receives none and $v_4$ receives $d_3$. Thus, if $v_1$ receives $t$ tags on the left, then $v_2$ gets $$ (\frac{dn}{2}v+d_1+d_3)-(\frac{dn}{2}v+d_3+t) =d_1-t.$$ 

By construction, $v_1$ must receive $\frac{dn}{2}-t$ tags on the right, while $v_2$ receives $\frac{dn}{2}-d_1+t$. Consider reversing the direction of all tags on the right side in this contribution. Then, $v_1$ receives $(\frac{dn}{2}-d_1)-(\frac{dn}{2}-t)=\frac{dn}{2}-d_1+t$ tags on the right, while $v_2$ receives $(\frac{dn}{2}-d_1)-(\frac{dn}{2}-d_1+t)=\frac{dn}{2}-t$. Further, $v_3$ receives all the tags on the right, and $v_4$ receives none. Exchanging edges on the right via the Schnetz twist, this then becomes a contribution to the permanent in $G_2^{[n]}$. Clearly, this is a bijection. Fixing an orientation in $G_1$ arbitrarily, and an orientation in $G_2$ by reversing the direction of all edges in the right side of the graph, we see that the values of all contributions are preserved. Hence, the permanents of these graphs are equal at this prime, and so the extended graph permanents must be equal. \end{proof}

Lastly, planar duals are known to have equal periods (\cite{Sphi4}). We now show that they also have equal extended graph permanents. The following results will be of use throughout the proof of the invariance of duality.

\begin{wilson} A natural number $n$ is prime if and only if $(n - 1)! \equiv −1 \pmod{n}$. \end{wilson}

\begin{proposition} For positive integers $s$ and $t$ such that $s>t$; 
\begin{align} \gcd (s,s-t) &= \gcd (s,t) \label{gcdstuff} \\
\frac{\lcm (s,t)}{t} &= \frac{\lcm (s,s-t)}{s-t} \label{lcmstuff} \\
  \frac{\lcm (s,t)}{t}&=\frac{\lcm (s,t)}{s} + \frac{\lcm (s,s-t)}{s}\label{anotherntthing}
  \end{align}  \end{proposition}

We may now prove that a graph and its dual have extended graph permanents defined on the same set of primes.

\begin{corollary} \label{primesaregood} Suppose $G=(V,E)$ is a connected planar graph such that $|V|>1$ and $|E| > |V|-1$. Let  $G^*=(V^*,E^*)$ be the planar dual of $G$. Then $\frac{\lcm(|V|-1,|E|)}{|V|-1} = \frac{\lcm(|V^*|-1,|E^*|)}{|V^*|-1}$ and the extended graph permanents for $G$ and $G^*$ are defined on the same set of primes. \end{corollary}

\begin{proof} By construction, $G^*$ has $|E|$ edges, and by Euler's polyhedral formula $2+|E|-|V|$ vertices. As the fundamental matrix $M_G$ has dimensions $\left( |V|-1 \right) \times |E|$, the extended graph permanent of $G$ is defined over primes of the form $\frac{\lcm (|V|-1, |E|)}{|V|-1}n+1$. The fundamental matrix $M_{G^*}$ is of dimension $(1+|E|-|V|) \times |E|$, and hence the extended graph permanent of $G^*$ is defined over primes of the form $\frac{\lcm (1+ |E| -|V| , |E|)}{1 + |E| -|V|}n+1$. By Equation~\ref{lcmstuff}, $ \frac{\lcm (|V|-1, |E|)}{|V|-1} = \frac{\lcm (1+ |E| -|V| , |E|)}{1 + |E| -|V|}$,  which completes the proof. \end{proof}

The following is Theorem 2.2.8 in \cite{Oxl}, translated in to graph theoretic language. 

\begin{theorem}\label{oxleydual} Let $G$ be a connected planar graph with $n$ edges, such that $G$ is neither a tree nor the empty graph. Order the edges of $G$ so that the first $r= |V(G)|-1$ form a spanning tree. Then, the reduced signed incidence matrix row reduces to $[I_r |A]$. Maintaining this ordering on the edges, the dual $G^*$ has reduced signed incidence matrix that can be row reduced to $[-A^T|I_{n-r}]$. \end{theorem}

This is key to proving that the extended graph permanent is indeed invariant under planar duals for $4$-point $\phi^4$ graphs. We require, then, that this row reduction does not change the permanents modulo the appropriate prime. We need to know then that we never must scale by a number other than $\pm 1$. To do this, we will use \emph{totally unimodular matrices}. These are matrices such that every square submatrix has determinant in $\{0,\pm1\}$. 

\begin{lemma}\label{tummy} A signed incidence matrix is totally unimodular.  \end{lemma}

\noindent A proof of this can be found in \cite{Oxl}.

Let $A = [a_{x,y}]$ be a matrix. Following Oxley, define \emph{pivoting} on entry $a_{s,t}$ as the series of operations used in standard Gaussian elimination to turn the $t^\text{th}$ column into the $s^\text{th}$ unit vector. 

\begin{proposition}[\cite{Oxl}] Let $A$ be a totally unimodular matrix. If $B$ is obtained from $A$ by pivoting on the non-zero entry $a_{s,t}$ of $A$, then $B$ is totally unimodular.  \end{proposition}

The following corollary is therefore immediate.

\begin{corollary}\label{constmultrowred} Row reducing as in Theorem~\ref{oxleydual}, we may choose a sequence of operations such that multiplication of a row by a constant only ever uses constant $- 1$. \end{corollary}

To work with these matrices, we will use cofactor expansion techniques. An important one is included in the following remark.

\begin{remark}\label{howitgoes} Let $M=[I_r|A]$ be a matrix and $\overline{M}= \mathbf{1}_{\V \times \E} \otimes M$ the fundamental matrix of $M$. Suppose we want to find the permanent of $$\mathbf{1}_n \otimes \overline{M} = \mathbf{1}_n \otimes (\mathbf{1}_{\V \times \E} \otimes M) = \mathbf{1}_{n\V \times n\E} \otimes M.$$

Note that there are $n\E$ copies of each column of $M$ in $\overline{M}$, and similarly $n\V$ copies of each row. Performing cofactor expansion along all copies of a column in the identity matrix blocks then produces a factor of $$n\V (n\V-1) \cdots (n\V -n\E +1) = \frac{(n\V)!}{(n\V - n\E)!}.$$ Each expansion removes one specific row of $A$, so in total $n\E$ copies of this row are removed. Over all columns in the identity matrix blocks, then, we get $$\text{Perm}(\mathbf{1}_n \otimes \overline{M}) = \left( \frac{(n\V)!}{(n\V - n\E)!} \right)^r \text{Perm}(\mathbf{1}_{(n\V -n\E) \times n\E} \otimes A).$$
\end{remark}

\begin{proposition}\label{dual} Suppose the graph $G$ is a connected, planar, and not a tree. Let $\V_G$ be the number of copies of each row in the fundamental matrix of $G$, and $\E_G$ the number of copies of each column. Suppose the reduced signed incidence matrix for $G$ is $M_G$, and row reduces to $[I_{|V(G)|-1}|A]$. For prime $n\V_G +1$, \begin{align*} \text{Perm}(\mathbf{1}_n \otimes \overline{M}_G) & \\ &\hspace{-2.5cm} \equiv  \left( \frac{(n\V_G)!}{(n\V_G - n\E_G)!} \right)^{|V(G)|-1} \text{Perm} (\mathbf{1}_{(n\V_G -n\E_G) \times n\E_G} \otimes A) \pmod{n\V_G +1}.  \end{align*}  \end{proposition}

\begin{proof}
By Remark~\ref{rowops}, Corollary~\ref{reduction}, and Corollary~\ref{constmultrowred}, row reduction operations preserve the extended graph permanent modulo $n\V_G +1$, as restricting to non-trees forces all primes to be odd, and hence an even number of repeated rows for all matrices. We may therefore row reduce $M_G$, the signed incidence matrix of $G$, to $ \left[ \begin{array}{c|c} I_{|V(G)|-1} & A \end{array} \right]$ by Theorem~\ref{oxleydual}. 

For prime $n\V_G+1$, by Remark~\ref{howitgoes}, this gives the desired residue. \end{proof}

Proposition~\ref{dual} translates to dual graphs quickly. For graph $G$ that meets the requirements and dual $G^*$, $|V(G^*)| = 2 - |V(G)| + |E(G)|$, and $|E(G^*)| = |E(G)|$. Therefore, \begin{align*} \text{Perm}(\mathbf{1}_n \otimes \overline{M}_{G^*}) &  \equiv \left( \frac{(n\V_{G^*})!}{(n \V_{G^*} - n\E_{G^*})!} \right)^{|V(G^*)|-1} \\ & \hspace{0.8cm} \cdot \text{Perm}(\mathbf{1}_{(n\V_{G^*} -n\E_{G^*}) \times n\E_{G^*}} \otimes -A^T) \pmod{n\V_G +1},\end{align*} using the reduction from Theorem~\ref{oxleydual}. By Corollary~\ref{primesaregood}, $\V_G = \V_{G^*}$, but $\E_G$ is not necessarily equal to $\E_{G^*}$.

\begin{corollary}[to Equation~\ref{lcmstuff}] \label{brokenlcms} For positive integers $s$ and $t$ such that $s>t$, $\lcm (s,t) = \lcm (s-t,s)$ if and only if $2t = s$. \end{corollary}

\begin{proof} \sloppy If $2t = s$, then $s-t = t$ and $\lcm (s-t,s) = \lcm (t,s)$. In the other direction, suppose $\lcm (s-t,s) = \lcm (t,s)$. By Equation~\ref{lcmstuff}, $\frac{\lcm (s-t,s)}{s-t} = \frac{\lcm (t,s)}{t}$, so $s-t = t$ and hence $2t=s$.   \end{proof}

As $\E_G = \frac{\lcm (|V(G)|-1, |E(G)|)}{|E(G)|}$, Corollary~\ref{brokenlcms} shows that $\E_G = \E_{G^*}$ if and only if $2(|V(G)|-1) = |E(G)|$. It follows that duality for $4$-point $\phi^4$ graphs is a special instance of general duality.

Let $G=(V,E)$ be a graph and $G^*=(V^*,E^*)$ its planar dual. Let $\overline{M}_G = \mathbf{1}_{\V_G \times \E_G} \otimes M_G$  and $\overline{M}_{G^*} = \mathbf{1}_{\V_{G^*} \times \E_{G^*}} \otimes M_{G^*}$ be respective fundamental matrices of these graphs. It follows from Equation~\ref{anotherntthing} and Proposition~\ref{primesaregood} that $\E_G + \E_{G^*} = \V_G = \V_{G^*}$.

\begin{lemma} \label{anotherwilsoncor} For $j=a+b+1$ where $a$ and $b$ are positive integers, $$a!\cdot b! \equiv (-1)^b(j-1)! \pmod{j}.$$  \end{lemma}

\begin{proof} Briefly, \begin{align*} a! \cdot b! &\equiv a! (b (b-1) \cdots 1) \\ & \equiv a! ( (b-j)(b-1-j) \cdots (1-j)) \\ &\equiv (1 \cdots a)((-1)^b(j-b)(j-b+1) \cdots  (j-1)) \\ & \equiv (-1)^b (j-1)! \pmod{j} \end{align*} as $j-b = a+1$.  \end{proof}

Using the previous notation, it follows from Lemma~\ref{anotherwilsoncor} that $$(n\E_G)!\cdot (n\E_{G^*})! \equiv (-1)^{n\E_G}(n\V_G)! \pmod{n\V_G+1}.$$ While we will generally be assuming that $n\V_G+1$ is prime and hence further simplification follows from Wilson's Theorem, we will be using this to simplify future calculations, and hence leave this computation here.

It is also worth briefly noting that for a graph $G$ and fundamental matrix $\overline{M}_G = \mathbf{1}_{\V_G \times \E_G} \otimes M_G$, the product $n\E_G|E(G)|$ is always even, where $n$ is an integer such that $p=n\V_G+1$ is an odd prime. If we suppose that $\E_G$ and $|E(G)|$ are both odd, then as $\E_G = \frac{\lcm(|E(G)|,|V(G)|-1)}{|E(G)|}$, it follows that $\lcm(|E(G)|,|V(G)|-1)$ is also odd. Thus, $\V_G = \frac{\lcm(|E(G)|,|V(G)|-1)}{|V(G)|-1}$ must be odd also. As $p$ is assumed to be an odd prime, $n$ must therefore be even.

\begin{proposition}\label{painful} Let $G=(V,E)$ be a graph and $G^*=(V^*,E^*)$ its planar dual, and suppose they have fundamental matrices $\overline{M}_G = \mathbf{1}_{\V_G \times \E_G} \otimes M_G$ and $\overline{M}_{G^*} = \mathbf{1}_{\V_{G^*} \times \E_{G^*}} \otimes M_{G^*}$. For common prime $p = n\V_G +1$, $$ \text{GPerm}^{[p]}(G) = (-1)^{|E|-|V|+1}(n\E_G)!^{|E|}\text{GPerm}^{[p]}(G^*)  .$$ \end{proposition}

\begin{proof} Computing the extended graph permanent of $G$ at prime $p$ and modulo $p$, \begin{align*} \text{GPerm}^{[p]}(G) &\equiv \text{Perm}(\mathbf{1}_n \otimes \overline{M}_G)  \\ 
&\equiv \left( \frac{(n\V_G)!}{(n\V_G - n\E_G)!} \right)^{|V|-1} \text{Perm}(\mathbf{1}_{(n\V_G - n\E_G) \times n\E_G} \otimes A)  \\ 
&\equiv \left( \frac{(n\V_G)! (n\E_G)!}{(n\E_{G^*})!(n\E_G)!} \right)^{|V|-1}\text{Perm}(\mathbf{1}_{(n\V_G - n\E_G) \times n\E_G} \otimes A)  \\
&\equiv (-1)^{n\E_G(|V|-1)}(n\E_G)!^{|V|-1}\text{Perm}(\mathbf{1}_{(n\V_G - n\E_G) \times n\E_G} \otimes A),  \end{align*} by Proposition~\ref{dual}, Equation~\ref{anotherntthing}, and Lemma~\ref{anotherwilsoncor}.  Similarly for $G^*$, and as $\V_G = \V_{G^*}$ by Corollary~\ref{primesaregood}, \begin{align*} \text{GPerm}(G^*) & \equiv \text{Perm}(\mathbf{1}_n \otimes \overline{M}_{G^*})  \\
&\hspace{-1.5cm} \equiv \left( \frac{ (n\V_G)!}{(n\V_{G^*} - n\E_{G^*})!} \right)^{|E|-|V|+1} \text{Perm} (\mathbf{1}_{(n\V_{G^*} - n \E_{G^*}) \times n \E_{G^*}} \otimes -A^T)  \\
&\hspace{-1.5cm}  \equiv \left( \frac{-1}{(n\E_G)!} \right)^{|E|-|V|+1} (-1)^{(|E|-|V|+1)n\E_G}\text{Perm}(\mathbf{1}_{n\E_G \times (n\V_G - n\E_{G})} \otimes A^T)   \\ 
&\hspace{-1.5cm}  \equiv \frac{(-1)^{(|E|-|V|+1)(n\E_G+1)}}{(n\E_G)!^{|E|-|V|+1}} \text{Perm} (\mathbf{1}_{(n\V_G - n\E_G) \times n \E_G} \otimes A)  \pmod{n\V_G+1}. \end{align*}  

Note the common factors in these two equivalences. Therefore, \begin{align*} \text{GPerm}^{[p]}(G) 
&\equiv  (-1)^{n\E_G(|V|-1)}(n\E_G)!^{|V|-1} (n\E_G)!^{|E|-|V|+1} \\ & \hspace{1cm} \cdot (-1)^{(|E|-|V|+1)(n\E_G+1)}\text{GPerm}^{[p]}(G^*) \\
&\equiv (-1)^{|E|-|V|+1}(n\E_G)!^{|E|}\text{GPerm}^{[p]}(G^*) \pmod{p}.\end{align*}
\end{proof}

To show specifically that this results in invariance for $4$-point $\phi^4$ graphs, we will use the following corollary to both Wilson's Theorem and Lemma~\ref{anotherwilsoncor}.

\begin{corollary}\label{wilsoncor} Let $p= 2n+1$ be an odd prime. Then, $$n!^2 \equiv \begin{cases} -1\pmod{p} \text{ if $n$ is even} \\ 1 \pmod{p} \text{ if $n$ is odd} \end{cases}.$$ \end{corollary}

\begin{corollary}\label{phi4dual} For planar graph $G=(V,E)$ where $2(|V(G)|-1) = |E(G)|$ and its planar dual $G^*$, $G$ and $G^*$ have equal extended graph permanents. \end{corollary}

\begin{proof} Here we consider primes of the form $p=2n+1$ for integers $n$. By Proposition~\ref{painful}, it suffices to consider only $(-1)^{|E| -|V|+1}(n\E_G)!^{|E|} \pmod{p}$. As $|E|=2(|V|-1)$ and $\E_G =1$, this is equivalent to $(-1)^{|V|-1}n!^{2(|V|-1)}$. 

If $n$ is even, then by Corollary~\ref{wilsoncor} $n!^2 \equiv -1 \pmod{p}$, and $$(-1)^{|V|-1}n!^{2(|V|-1)} \equiv (-1)^{2(|V|-1)} \equiv 1 \pmod{p} .$$ Otherwise, $n!^2 \equiv 1 \pmod{p}$, and $$ (-1)^{|V|-1}n!^{2(|V|-1)} \equiv (-1)^{|V|-1} \pmod{p} .$$ As these are primes for which the extended graph permanent vales may vary based on the underlying orientation of the directed graph, this produces either equivalence or a constant sign difference that can be corrected by reversing the direction of one edge. \end{proof}

While not period preserving itself, decompleted graphs with $2$-vertex cuts also have an important property with regards to the period. Breaking the graph as in Figure \ref{2cut} and assuming all are $4$-point $\phi^4$, the period of $G$ is equal to the products of the periods of $G_1$ and $G_2$. As such, we would like for the extended graph permanent to have this property also. Before we can prove that it does, we require the following useful result.

\begin{figure}[h]
  \centering
      \includegraphics[scale=1.20]{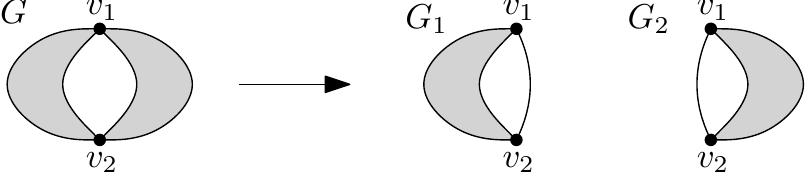}
  \caption{Operation on a $2$-vertex cut. If all are $4$-point graphs in $\phi^4$ theory, then the period of $G$ is equal to the product of the periods of $G_1$ and $G_2$.}
\label{2cut}
\end{figure}

\begin{lemma}[Lemma 21 in \cite{crump}]\label{pigeon} Suppose $M = \left[ \begin{array}{cc} A & \bf{0} \\ \bf{0} & B \end{array} \right]$ is a square block matrix, where $A$ and $B$ are arbitrary and $\bf{0}$ is all-zero. If $A$ is not square then the permanent of $M$ is zero.  \end{lemma}

\begin{theorem} \label{2vertexcut} Consider the graph $G$ and two minors $G_1$ and $G_2$ seen in Figure~\ref{2cut}. If for all $G' \in \{G,G_1,G_2\}$, $2|V(G')| -2= |E(G')|$, then for all odd primes $p$  $\text{GPerm}^{[p]}(G) = - \text{GPerm}^{[p]}(G_1) \text{GPerm}^{[p]}(G_2).$ \end{theorem}

\begin{proof} Set $v_2$ as the special vertex for all graphs, and write $(C|D)$ as the row corresponding to vertex $v_1 \in V(G)$. Then, we have signed incidence matrices 
$$M_G = \left[ \begin{array}{ccc|ccc}
& G_1 & & & \bf{0}& \\ 
&C& && D& \\ 
 & \bf{0} & & & G_2 & \end{array} \right], $$  

 $$M_{G_1} = \left[ \begin{array}{ccc|c}
& G_1& & \bf{0} \\  
&C& & 1\end{array} \right], $$ and 

 $$M_{G_2} = \left[ \begin{array}{c|ccc}
\bf{0} && G_2 &  \\ 
1 &&D&  \end{array} \right]. $$ By extension then, the fundamental matrices are $2$-matrices, and we want to compute permanents for $\mathbf{1}_{2k \times k} \otimes M$ for $M \in \{ M_G, M_{G_1}, M_{G_2} \}$.

Computing the permanent for $G$ by cofactor expansion along $2k$ rows $(C|D)$ and using Lemma \ref{pigeon}, the remaining blocks will only be square if $k$ columns are taken from the edges in $G_1$ and $k$ from edges in $G_2$. We use notation $N_S$ to denote matrix $N$ with a set of columns $S$ removed. Further, we assume the edges are oriented so that all entries in rows $(C|D)$ are in $\{ 0,1\}$. For notational convenience take $\mathcal{C}$ as the set of non-zero columns of $\mathbf{1}_{1 \times k} \otimes (C|D)$ that are in $C$, and similarly let $\mathcal{D}$ be the set of non-zero columns of $\mathbf{1}_{1 \times k} \otimes (C|D)$ in $D$. Hence, \begin{align*} 
\text{Perm} \left( \mathbf{1}_{2k \times k} \otimes M_G \right)
&= (2k)! \sum_{\substack{i_1,...,i_k \in \mathcal{C} \\ j_1 , ..., j_k \in \mathcal{D}}} \text{Perm} \left( \mathbf{1}_{2k \times k} \otimes \left[ \begin{array}{cc} G_1&\bf{0} \\ \bf{0}&G_2 \end{array} \right]_{ \substack{ \{i_1,...,i_k, \\ \hspace{3mm}j_1,...,j_k \} }  } \right) \\ 
&= (2k)! \sum_{i_1,...,i_k \in \mathcal{C}} \text{Perm} \left( \mathbf{1}_{2k \times k} \otimes \left[ \begin{array}{c} G_1 \end{array} \right]_{\left\{i_1,...,i_k \right\}} \right) \\
&\hspace{1cm} \times \sum_{j_1 , ..., j_k \in \mathcal{D}} \text{Perm} \left( \mathbf{1}_{2k \times k} \otimes \left[ \begin{array}{c} G_2 \end{array} \right]_{\left\{ j_1,...,j_k \right\} } \right) . \end{align*} We assume in all summations that elements of $i_1,...,i_k$ and $j_1,...,j_k$ are pairwise disjoint.

Similarly, expanding along the $k$ columns corresponding to the new edges in $G_1$ and then the $k$ remaining rows corresponding to $(C)$, we get;
\begin{align*} 
\text{Perm} \left( \mathbf{1}_{2k \times k} \otimes M_{G_1} \right) 
&= \frac{(2k)!}{k!} \text{Perm} \left( \mathbf{1}_k \otimes \left[ \begin{array}{c} G_1 \\ G_1 \\ C  \\ \end{array} \right] \right) \\ 
&= \frac{(2k)!}{k!} k! \sum_{ i_1,...,  i_k \in \mathcal{C}} \left( \mathbf{1}_k \otimes  \text{Perm} \left[ \begin{array}{c} G_1 \\ G_1 \end{array} \right]_{ \{i_1,...,i_k\}} \right).  \end{align*} 
Similarly, $$\text{Perm}(\mathbf{1}_{2k \times k} \otimes M_{G_2})= (2k)! \sum_{j_1,...,j_k \in \mathcal{D}} \text{Perm} \left( \mathbf{1}_k \otimes \left[ \begin{array}{c} G_2 \\ G_2 \end{array} \right]_{\{ j_1,...,j_k\} } \right).$$ As $(2k)! \equiv -1 \pmod{2k+1}$  by Wilson's Theorem, the extended graph permanents differ by a constant sign.
  \end{proof}

Given the collection of theorems in this section, it is natural to make the following conjecture.

\begin{conjecture}\label{obviousconjecture} If two $\phi^4$ graphs have equal period, then they have equal extended graph permanent. \end{conjecture}

\noindent As our motivation for creating this invariant was its potential to help understand the period, the data suggests that the connection is in fact there.

\section{Computation of the extended graph permanent}
\label{egpcomp}

The permanents of large matrices are notoriously difficult to compute; the lack of row-reduction techniques mean that usually computations are done using the definition or cofactor expansion. However, as we desire only the residue, we can use row reduction, provided we have not prior used cofactor expansion to reduce the number of identical blocks. Further, our matrices are constructed with a great deal of repetition, which results in easier cofactor expansion. In this section, we simplify the computation of the extended graph permanents, and produce closed forms for several graph families. We do this using standard combinatorial counting techniques and cofactor expansion. 

To emphasize the structural nature of our cofactor expansion, we will represent the permanents of $k$-matrices as weighted graphs, weights on edges  (vertices) counting the number of columns (rows) appearing in the matrix that represent that edge (vertex). Since we are representing the permanent graphically, we will differentiate from graphs by writing these weighted graph representations of the permanent in square brackets.

Representations of this type are not unique. If a graph has multiple vertices of weight zero, those vertices are indistinguishable, as they correspond to rows that do not occur in the matrix. However, up to reordering the rows and columns, the graphical representation does uniquely produce a matrix. Trivially, the matrix must be square if we are to take a permanent, and hence we require that the sum of the vertex weights must be equal to the sum of the edge weights.

\begin{example} \begin{align*} \left[ \raisebox{-.48\height}{\includegraphics{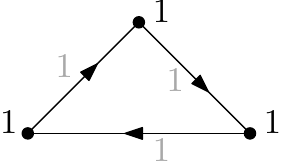}} \right] 
&= \text{Perm} \left[ \begin{array}{ccc} 1 & 0 & -1 \\ -1 & 1 & 0 \\ 0 & -1 & 1 \end{array} \right]\\
\left[ \raisebox{-.48\height}{\includegraphics{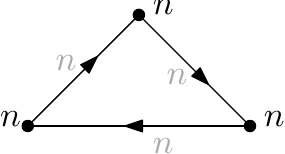}} \right] 
&= \text{Perm} \left( \mathbf{1}_n \otimes \left[ \begin{array}{ccc} 1 & 0 & -1 \\ -1 & 1 & 0 \\ 0 & -1 & 1 \end{array} \right] \right) 
 \end{align*} \end{example}

There is a general method for writing the cofactor expansion that occurs at vertices using this method. Suppose that vertex $v$ has weight $w_v \neq 0$, and further that $n$ incident edges $e_1 = (v,v_1),...,e_n = (v,v_n)$ have weights $w_1, ..., w_n$. Let $m_{e_i}$ denote the value in the matrix of edge $e_i$ at vertex $v$. Performing cofactor expansion along all rows corresponding to vertex $v$, 
\begin{align*}
\left[ \raisebox{-.48\height}{\includegraphics[scale=0.6]{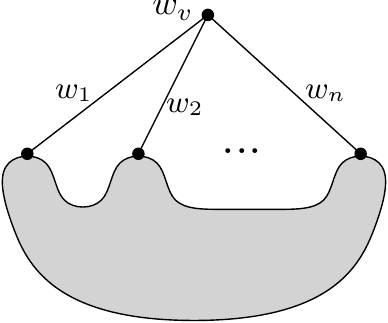}} \right] = \sum_{\substack{k_1 + \cdots + k_n = w_v\\ k_i \geq 0}} w_v! \prod_{j=1}^n \binom{w_j}{k_j}  m_{e_j}^{k_j} \left[ \raisebox{-.48\height}{\includegraphics[scale=0.6]{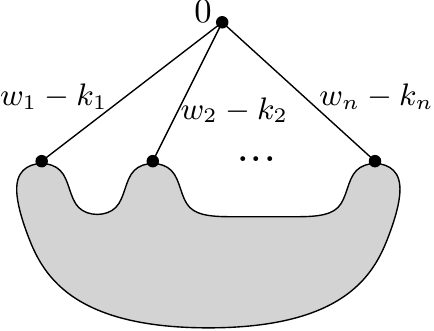}} \right].
\end{align*}
The $w_v!$ factor comes from the fact that order matters in the selection of edges.

%An alternate, and more graphic, interpretation of the cofactor expansion relies on selection of the edges first. That is, after selection of edges via multinomial, cofactor expansion along a row designated to select edge $e_i$ meets that edge $w_i$ times in the first expansion, $w_i-1$ times in the second expansion, and so on. This gives \begin{align*}\left[ \raisebox{-.48\height}{\includegraphics[scale=0.6]{genex1}} \right] = \sum_{k_1 + \cdots + k_n = w_v} \binom{w_v}{k_1,...,k_n} \frac{w_1! v_1^{k_1}}{(w_1-k_1)!} \cdots \frac{w_n! v_n^{k_n}}{(w_n-k_n)!}\left[ \raisebox{-.48\height}{\includegraphics[scale=0.6]{genex2}} \right].\end{align*} It is a trivial check that these are equal.

One may also do cofactor expansion along a column, which corresponds to an edge. Herein, for algorithmic simplicity we will only use edges when the weight on one vertex is zero. Let $m_{v_i}$ be the value in the matrix at edge $e_i$ and vertex $v_i$. Then, with weights $w_i$ and $x_i$,\begin{align*}
\left[ \raisebox{-.48\height}{\includegraphics[scale=0.6]{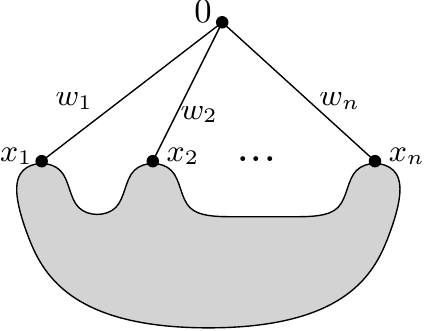}} \right]
&=  \prod_{i=1}^n \frac{x_i!}{(x_i-w_i)!} m_{v_i}^{w_i} \left[ \raisebox{-.48\height}{\includegraphics[scale=0.6]{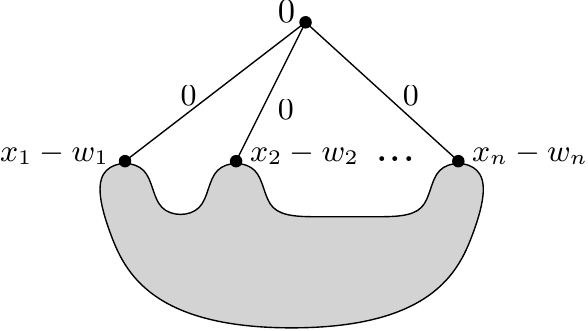}} \right]\\
&= \prod_{i=1}^n \frac{x_i!}{(x_i-w_i)!} m_{v_i}^{w_i} \left[ \raisebox{-.48\height}{\includegraphics[scale=0.6]{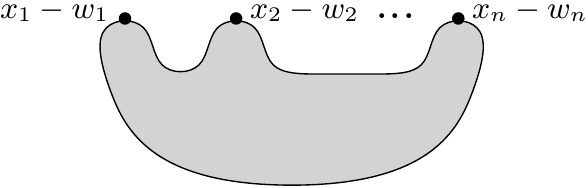}} \right].
\end{align*} This last line follows from the fact that an edge with weight zero contributes nothing to the matrix and is hence removable. Similarly, a vertex with weight zero and all incident edges having weight zero can be removed.

While orientations are ultimately arbitrary, we will include directions on edges to make the computations easier to follow. We will generally only apply the orientation when we are about to act upon that edge or an incident vertex, purely for the sake of simplicity in the figures.

\subsection{Trees}\label{treesnshit}

Immediately, the signed incidence matrix of a tree with a row deleted, $M$, will give a square matrix, and hence $M = \overline{M}$. As such, we are interested in $\mathbf{1}_n \otimes M$ for all primes $p=n+1$. Applying Wilson's Theorem to a minimal non-trivial tree, 
$$ \left[ \raisebox{-.48\height}{\includegraphics{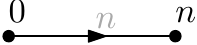}} \right] = \text{Perm} \left[ \mathbf{1}_{n} \right] = n! \equiv -1 \pmod{n+1}.$$ Note that the $4$-point $\phi^4$ graph $P_{1,1}$, the unique loop-free graph with two vertices and two edges, falls into this case. As $n$ will be even after prime two, the duplicated-edges view of the permanent is agnostic to one edge duplicated $n = 2k$ times or two edges in parallel duplicated $k$ times.

For general trees, we progress inductively. As any tree $T$ with at least two vertices starts with the special vertex having weight $0$ and all edges and non-special vertices with weight $n$, we will assume that the special vertex was a leaf. Hence, \begin{align*} \left[ \raisebox{-.48\height}{\includegraphics{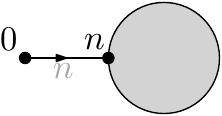}} \right] &= n! \left[ \raisebox{-.48\height}{\includegraphics{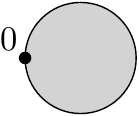}} \right]. \end{align*} This second figure represents the $n+1^\text{st}$ graph permanent of a smaller tree. Hence, we may move the special vertex again to a leaf. With base case established prior, we get that $\text{GPerm}^{[p]}(T) = (-1)^{|V(T)|-1} \pmod{p}$.

\subsection{Wheels}\label{wheels} An important family of graphs is the wheels, built from cycles by adding an apex vertex. Consider a wheel with $w$ vertices in the outer cycle, call it $W_w$. While only $W_3$ and $W_4$ are $4$-point $\phi^4$ graphs, all have $|E(W_w)| = 2(|V(W_w)|-1)$, and hence have extended graph permanent sequences built over all odd primes.

For prime $2n+1$, 
\begin{align*} \left[ \raisebox{-.48\height}{\includegraphics{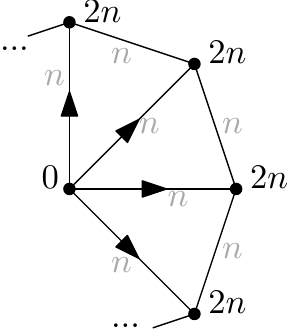}} \right] &= \left( \frac{(2n)!}{n!} \right)^w \left[  \raisebox{-.48\height}{\includegraphics{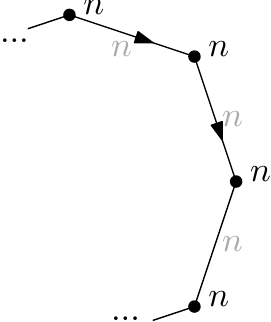}} \right] \\
&\hspace{-1.5cm}= \left( \frac{(2n)!}{n!}\right)^w \sum_{k=0}^n \binom{n}{k} \binom{n}{n-k} n!(-1)^k  \left[  \raisebox{-.48\height}{\includegraphics{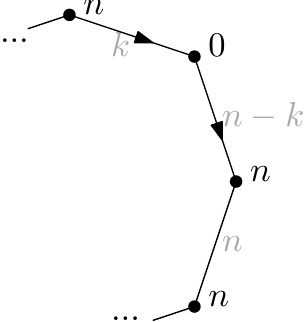}} \right] \\
&\hspace{-1.5cm}=  \left( \frac{(2n)!}{n!} \right)^w \sum_{k=0}^n (-1)^n \binom{n}{k}^2 n! \frac{n!}{k!} \frac{n!}{(n-k)!} \left[ \hspace{-1mm} \begin{array}{c} w-1 \\ \text{vertices} \end{array} \hspace{-1mm} \left\{  \raisebox{-.48\height}{\includegraphics{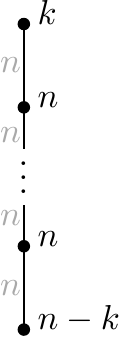}}  \right.  \right] . \\
&\hspace{-1.5cm}=  \left( \frac{(2n)!}{n!} \right)^w \sum_{k=0}^n (-1)^n \binom{n}{k}^3 n!^2 \left[ w-1 \left\{  \raisebox{-.48\height}{\includegraphics{wheel3}}  \right.  \right] .
\end{align*}

We pause in this calculation to consider the permanent of the path created. We will orient all edges away from the middle of the path. We then get; 
\begin{align*}
\left[ w-1 \left\{  \raisebox{-.48\height}{\includegraphics{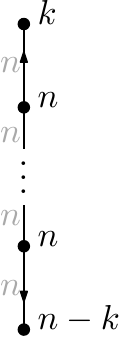}} \right. \right] 
&=  \frac{n!}{(n-k)!} \frac{n!}{k!} \left[ w-1 \left\{  \raisebox{-.48\height}{\includegraphics{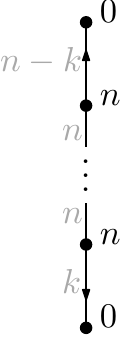}} \right. \right] \\
&\hspace{-.6cm}=  \binom{n}{k}^2 n!^2 (-1)^{n-k}(-1)^k \left[ w-3 \left\{  \raisebox{-.48\height}{\includegraphics{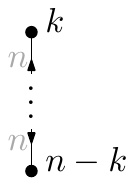}} \right. \right] \\
&\hspace{.75cm} \vdots \\
&\hspace{-.6cm}=  \binom{n}{k}^{w-3}  n!^{w-3} (-1)^{n \lfloor \frac{w-3}{2}  \rfloor}
\begin{cases} \left[  \raisebox{-.48\height}{\includegraphics{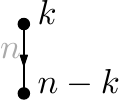}} \right] \vspace{1mm} & \text{ if $w-1$ is even} \\  
\left[  \raisebox{-.48\height}{\includegraphics{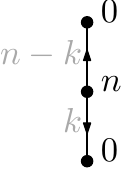}} \right] & \text{ if $w-1$ is odd}
\end{cases} \\
&\hspace{-.6cm}=   \binom{n}{k}^{w-3}  n!^{w-3} (-1)^{n \lfloor \frac{w-3}{2}  \rfloor}
\begin{cases}(-1)^k n! & \text{ if $w-1$ is even} \\  (-1)^n n! & \text{ if $w-1$ is odd}
\end{cases}.
\end{align*}

We continue with the original calculation;
\begin{align*}
 \left[ \raisebox{-.48\height}{\includegraphics{wheel1}} \right] 
&\equiv  \left( \frac{(2n)!}{n!} \right)^w \sum_{k=0}^n (-1)^n \binom{n}{k}^3 n!^2 \left[ w-1 \left\{  \raisebox{-.48\height}{\includegraphics{wheel3}}  \right.  \right]\\
&\hspace{-1.2cm}\equiv  \left( \frac{-1}{n!} \right)^w \sum_{k=0}^n (-1)^n \binom{n}{k}^w n!^w \cdot \begin{cases} (-1)^{k+n\lfloor \frac{w-3}{2}\rfloor} \text{ if $w$ is odd} \\ (-1)^{n+n\lfloor \frac{w-3}{2}\rfloor} \text{ if $w$ is even}\end{cases} \\
&\hspace{-1.2cm}\equiv \begin{cases}  (-1)^{w+n\lfloor \frac{w-1}{2} \rfloor} \sum_{k=0}^n (-1)^k \binom{n}{k}^w \text{ if $w$ is odd} \\ (-1)^{w+n \lfloor \frac{w+1}{2} \rfloor}  \sum_{k=0}^n \binom{n}{k}^w \text{ if $w$ is even} \end{cases} \\
&\hspace{-1.2cm}\equiv (-1)^{w + n\lceil \frac{w-1}{2} \rceil} \sum_{k=0}^n (-1)^{kw}\binom{n}{k}^w  \pmod{2n+1}. 
\end{align*}

As a factor $(-1)^n$ corresponds to reversing the direction of the $n$ columns corresponding to a common edge, we may write this as $$(-1)^{w} \sum_{k=0}^n (-1)^{kw}\binom{n}{k}^w  \pmod{2n+1}.$$

It is interesting to note that, if $2n+1$ is congruent to three modulo four and hence $n$ is odd, then \begin{align*}  \sum_{k=0}^n (-1)^k \binom{n}{k}^w &= \binom{n}{0}^w - \binom{n}{1}^w + \cdots - \binom{n}{n}^w \\ &= \sum_{k=0}^{(n-1)/2} \left( \binom{n}{k}^w - \binom{n}{n-k}^w \right) =0. \end{align*} This vanishing permanent property generalizes to graphs with a particular symmetry. To prove this generalized equality, we require the more graphical interpretation of the permanent calculation from Remark \ref{altrem}.

\begin{definition} Let $G$ be a graph. If graph automorphism $\tau$ has $\tau(\tau(v)) =v$ for all $v \in V(G)$, then $\tau$ is an \emph{involution}. For a particular involution $\tau$, we will say that an edge $e =uv$ is \emph{crossing} if $\tau(u)=v$. \end{definition}

\begin{theorem} \label{symmetry} Suppose $G$ is a graph. If there is an involution $\tau$ with an odd number of crossing edges and at least one vertex fixed by $\tau$, then the permanent of the associated $k$-matrix for $G$ is identically zero. \end{theorem}

\begin{proof} Set a vertex fixed by $\tau$ as the special vertex. For non-crossing edges $e=uv$, orient such that the involution preserves the orientation; if $e = (u,v)$ then $\tau(u)\tau(v)= (\tau(u),\tau(v))$.  Finally, orient the crossing edges arbitrarily.

Valid edge colourings and taggings are preserved by the automorphism. Sign changes occur only when an odd number of tags change direction, and hence only due to crossing edges. We may therefore partition all taggings into two sets by fixing a crossing edge and dividing the taggings based on which vertex incident to this edge received the tag. As the automorphism provides an obvious bijection between sets, the sum is zero. \end{proof}

\begin{corollary}\label{symmetrycor} \sloppy If a decompleted $4$-regular graph admits an involution with an odd number of crossing edges and at least one fixed vertex, then $\text{GPerm}^{[p]}(G) = 0$ for all primes $p \equiv 3 \pmod{4}$.\end{corollary}

\begin{proof} If $G$ has an odd number of crossing edges, then so does $G^{[k]}$ for odd $k$, and the result follows. \end{proof}

All wheels $W_k$ for odd $k$ have such involutions. Computing up to prime $p=4999$, this actually explains all zeros in the extended graph permanent of $W_3$. Wheel $W_5$ has $\text{GPerm}^{[5]}(W_5) \equiv 0 \pmod{5}$, though the actual permanent of the associated matrix is non-zero, so this does not explain all zeros up to residues.

The graph shown in Figure \ref{p711}, named $P_{7,11}$ in  \cite{Sphi4}, is a counterexample to the converse of Theorem \ref{symmetry}. The figure is drawn so that the marked symmetry captures the only element of the symmetric group, which has an even number of crossing edges and no fixed vertices. Choosing the grey vertex as special, the permanent of the signed incidence $2$-matrix is equal to zero. Up to prime $p=199$, this is the only identically zero permanent in the sequence.

\begin{figure}[h]
  \centering
      \includegraphics[scale=1.0]{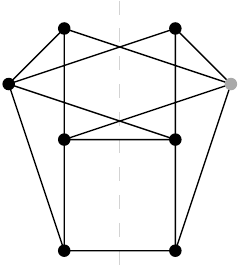}
  \caption{A graph that provides a counterexample to the converse of Theorem \ref{symmetry}.}
\label{p711}
\end{figure}

We see a number of graphs with sequences containing zeroes at all primes congruent to three modulo four in Appendix \ref{chartofgraphs}. All of these graphs have a decompletion of the type described above, and hence all are explained by Corollary~\ref{symmetrycor}.

\subsection{Zig-zag graphs}

The zig-zag graphs are an important family in $\phi^4$ theory. The family has known closed form period (see \cite{BSZigzag}). Graphically, it is a family whose completions are \emph{circulant graphs} $C^n_{a,b}$; $$V(C^n_{a,b}) = \{0,...,n-1\}, \hspace{2mm} E(C^n_{a,b}) = \left\{ \{i,j\} : |i-j| \in \{a,b\pmod{n}\}  \right\},$$ where $a=1$, $b=2$. 

To generalize, consider the zig-zag graph on $m$ vertices, $m \geq 4$, as seen in Figure \ref{zarb02}. We will take the right-most vertex as the special vertex, and for the sake of future row reduction use the edges highlighted as the first $m-1$ columns in the signed incidence matrix. As such, our signed incidence matrix is 
\begin{align*} \left[
\begin{array}{cccccc|ccccccc}
1&0 &0 & &0 &0 &1 &0 &0 &0 & &0 &1 \\
-1&1 &0 & &0 &0 &0 &1 &0 &0 && 0 &0 \\
0&-1 &1 & &0 &0 &-1 &0 &1 &0 & &0 &0 \\
0&0 &-1 & &0 &0 &0 &-1 &0 &1 & &0 &0 \\
& & &\ddots & & & & & & &\ddots & & \\
0&0 &0 & &1 &0 &0 &0 &0 &0 & &1 &0 \\
0&0 &0 & &-1 &1 &0 &0 &0 &0 & &0 &0
\end{array} \right]_{m-1,2(m-1)}. \end{align*}
We may reduce this matrix, since we will be taking the permanent modulo $2n+1$. Hence, it reduces to
\begin{align*} \left[
\begin{array}{cccccc|ccccccc}
1&0 &0 & &0 &0 &1 &0 &0 &0 & &0 &1 \\
0&1 &0 & &0 &0 &1 &1 &0 &0 && 0 &1 \\
0&0 &1 & &0 &0 &0 &1 &1 &0 & &0 &1 \\
0&0 &0 & &0 &0 &0 &0 &1 &1 & &0 &1 \\
& & &\ddots & & & & & & &\ddots & & \\
0&0 &0 & &1 &0 &0 &0 &0 &0 & &1 &1 \\
0&0 &0 & &0 &1 &0 &0 &0 &0 & &1 &1
\end{array} \right]. \end{align*}

Label this right block $A$. Then, the matrix used for prime $2n+1$ in the extended graph permanent is $\mathbf{1}_{2n \times n} \otimes \left[ I_{m-1} | A \right]$. Cofactor expansion along the columns in the identity matrix gives $$ \text{Perm} (\mathbf{1}_{2n \times n} \otimes \left[ I_{m-1} | A \right]) = \left( \frac{(2n)!}{n!} \right)^{m-1} \cdot \text{Perm}(\mathbf{1}_n \otimes A) .$$

\begin{figure}[h]
  \centering
      \includegraphics[scale=0.90]{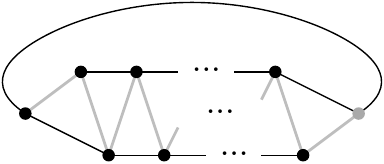}
  \caption{The general zig-zag graph, used to build the signed incidence matrix.}
\label{zarb02}
\end{figure}

What we see with this matrix $A$ is familiar; it is the incidence matrix of an undirected path on $m-1$ vertices, with one additional hyper-edge that meets all vertices. In our terms, all edges and vertices receive weight $n$, including the hyper-edge. Using cofactor expansion along the columns corresponding to the hyper-edge followed by our usual tricks;

\begin{align*}
\text{Perm}(I_n \otimes A) &= \sum_{\substack{k_1+\cdots + k_{m-1} = n \\ k_i \geq 0}} \binom{n}{k_1} \cdots \binom{n}{k_{m-1}} n! \left[ \raisebox{-.48\height}{\includegraphics[scale=.7]{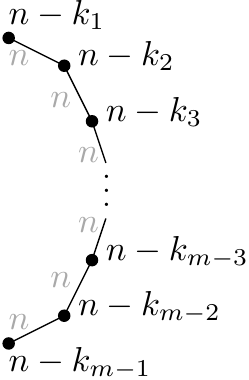}} \right], \end{align*}
where
\begin{align*} \left[ \raisebox{-.48\height}{\includegraphics[scale=.7]{zarb03}} \right] &= \frac{n!}{k_1!} \left[ \raisebox{-.48\height}{\includegraphics[scale=.7]{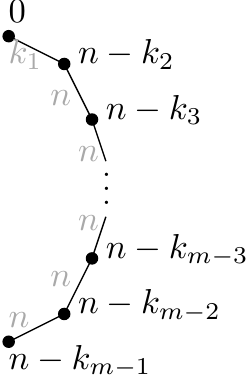}} \right] \\
&\hspace{-1cm}= \frac{n!}{k_1!} \frac{(n-k_2)!}{(n-k_1-k_2)!} \left[ \raisebox{-.48\height}{\includegraphics[scale=.7]{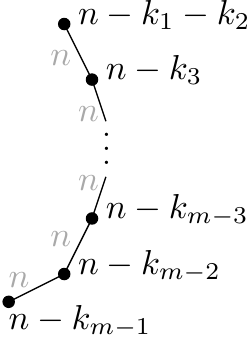}} \right] \\
&\hspace{-1cm}= \frac{n!}{k_1!} \frac{(n-k_2)!}{(n-k_1-k_2)!} \frac{n!}{(k_1+k_2)!} \left[ \raisebox{-.48\height}{\includegraphics[scale=.7]{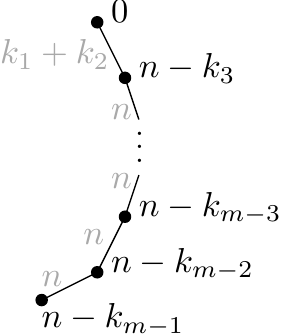}} \right] \\
&\hspace{-1cm}= \frac{n!}{k_1!} \frac{(n-k_2)!}{(n-k_1-k_2)!} \frac{n!}{(k_1+k_2)!}\frac{(n-k_3)!}{(n-k_1-k_2-k_3)!} \left[ \raisebox{-.48\height}{\includegraphics[scale=.7]{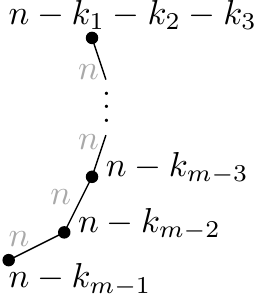}} \right] \\
&\vdots \\
&\hspace{-1cm}= \frac{n!}{k_1!} \frac{(n-k_2)!}{(n-k_1-k_2)!} \cdots \\
&\hspace{.3cm} \frac{n!}{(k_1+ \cdots +k_{m-3})!} \frac{(n-k_{m-2})!}{(n-k_1-\cdots - k_{m-2})!}\left[ \raisebox{-.48\height}{\includegraphics[scale=.6]{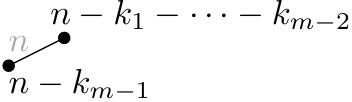}} \right] \\
&\hspace{-1cm}= \frac{n!}{k_1!} \frac{(n-k_2)!}{(n-k_1-k_2)!} \cdots \frac{n!}{(k_1+ \cdots +k_{m-3})!} \frac{(n-k_{m-2})!}{(n-k_1-\cdots - k_{m-2})!}n!\\
&\hspace{-1cm}= (n!)^{m-2} \binom{n-k_2}{k_1} \binom{n-k_3}{k_1+k_2} \cdots \binom{n-k_{m-2}}{k_1+ \cdots + k_{m-3}}.
\end{align*}

Hence, 
\begin{align*}
 \text{Perm} (\mathbf{1}_{2n \times n} \otimes \left[ I_{m-1} | A \right]) & \\
&\hspace{-3.4cm}= \left( \frac{(2n)!}{n!} \right)^{m-1} \sum_{\substack{k_1+\cdots + k_{m-1} = n \\ k_i \geq 0}} n!^{m-1} \left( \prod_{i=1}^{m-1} \binom{n}{k_i} \prod_{i=1}^{m-3} \binom{n-k_{i+1}}{\sum_{j=1}^i k_j} \right) \\
&\hspace{-3.4cm}= (2n)!^{m-1} \sum_{\substack{k_1+\cdots + k_{m-1} = n \\ k_i \geq 0}}  \left( \prod_{i=1}^{m-1} \binom{n}{k_i} \prod_{i=1}^{m-3} \binom{n-k_{i+1}}{\sum_{j=1}^i k_j} \right)\\
&\hspace{-3.4cm} \equiv (-1)^{m-1} \sum_{\substack{k_1+\cdots + k_{m-1} = n \\ k_i \geq 0}} \left( \prod_{i=1}^{m-1} \binom{n}{k_i} \prod_{i=1}^{m-3} \binom{n-k_{i+1}}{\sum_{j=1}^i k_j} \right) \pmod{2n+1}.
\end{align*}

\subsection{Computational Simplicity}
\label{niceness} 

In this subsection, we expressly forbid graphs with loops or vertices with precisely one neighbour.

It is possible to simplify the techniques introduced in this section to produce these closed forms in a more algorithmic way for individual graphs. This is due to the cancellation of terms that occurs, and an algorithmic way of gathering the factorials into binomials. We present this method here.

We will say that we \emph{act on} a vertex when we perform cofactor expansion on the collection of rows corresponding to that vertex. Similarly, we will say we act on an edge by performing cofactor expansion on the associated set of columns. We will similarly say we \emph{move to} a vertex or edge when we act on an incident object.

\sloppy Let $G$ be a graph. As before, let $L = \lcm (|V(G)|-1, |E(G)|)$, $\V = \frac{L}{|V(G)|-1}$, and $\E = \frac{L}{|E(G)|}$, so that for prime $\V n+1$ non-special vertices receive weight $\V n$ and edges receive weight $\E n$. Suppose we produce a closed form for $G$ by first acting on the edges incident to the special vertex, and then acting on a set of vertices, and their incident edges, such that the deletion of these vertices produces a tree. Finally, act on vertices of degree one and incident edges until the computation is complete.

For each non-special vertex, we therefore produce a factor of $(\V n)!$ in the numerator. Further, if vertex $v \in V(G)$ is not a vertex that is acted on initially in the production of a tree, most factors produced for this vertex cancel; every time we move to $v$ from an incident edge we produce a factorial in the numerator equal to the current weight, and a factorial in the denominator equal to the weight after this action. Thus, repeated actions telescope, and only the first weight remains in the numerator, and only the last weight remains in the denominator.

For edges, this holds true also. Each edge in the initial set of actions on vertices produces a binomial in $\E n$. All remaining edges will produce a factor of $(\E n)!$ in the numerator, and $(\E n - w_v)!$ in the denominator, where $w_v$ is the weight of the incident vertex. It follows then that these terms may be gathered into further binomials in $\E n$, as there are terms $\frac{(\E n)!}{w_v! (\E n - w_v)!} $.

As such, the graph produces a closed form that contains only factors $(\V n)!^{|V(G)|-1}$, summations from $0$ to $\E n$, a factor of $-1$ to some power, and a set of binomials in $\E n$, the number of these equal to the number of edges not incident to the special vertex. Further, these binomials come in two distinct flavours; those coming from an initial action on a vertex prior to establishing a tree, and those of the form $\binom{\E n}{w_v}$ where $w_v$ is the last weight on a vertex.

\begin{example}  Consider the wheel on four spokes, $W_4$. We saw in Section~\ref{wheels} that setting the apex vertex as the special vertex, we produce closed form $(2n)!^4\sum_{x=0}^n \binom{n}{x}^4$. Using the method described in this section, we can compute this using the following sequence. We include the variable distribution corresponding to acting on the vertex of degree two for clarity, and colour the vertex being acted on grey. This produces a closed form as follows.
$$\includegraphics[width=\textwidth]{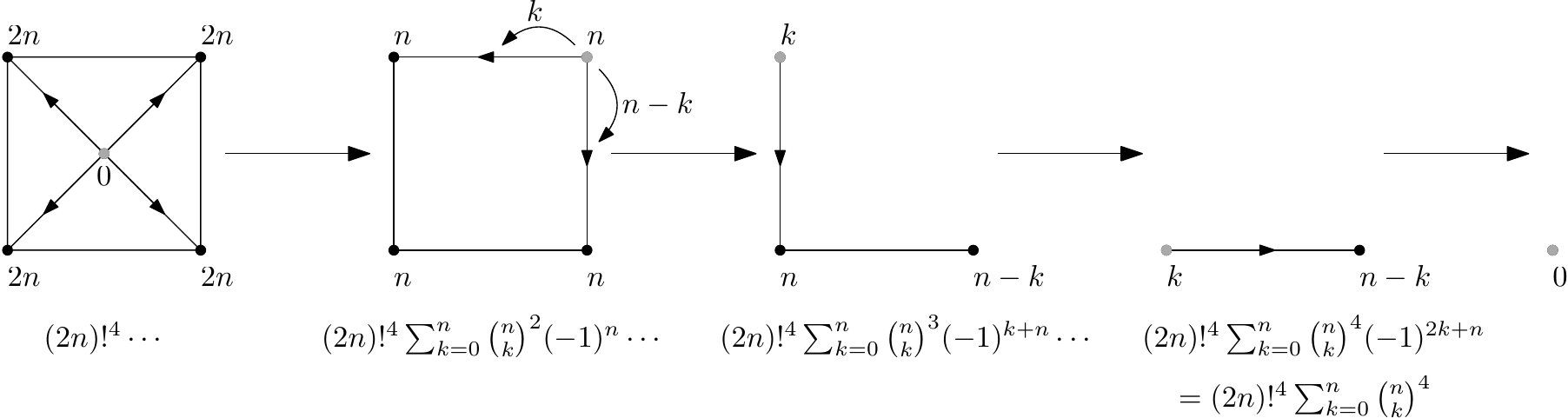}$$

If we instead choose a different special vertex, we may use the following sequence.
$$\includegraphics[width=\textwidth]{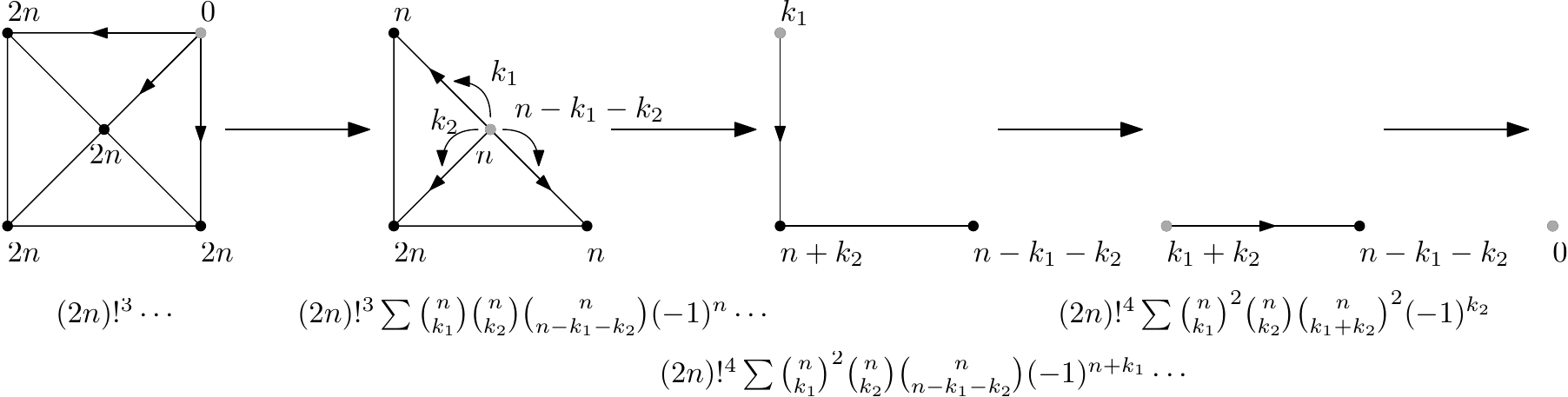}$$

 While these forms do not appear similar, and certainly are not equal prior to taking residues, we know from the previous work that the residues must agree modulo $2n+1$. Traditional methods, such as the Wilf-Zeilberger algorithm, generally will not work in explaining equalities such as this, as here we are comparing summations modulo $p$, and the permanent itself is affected by choice of special vertex. \end{example}

This method produces computationally easy formulas, though not necessarily the fastest. Computations using methods from earlier in this section may require careful consideration of the bounds in the summations, as otherwise negative factorial term may appear and break computer calculations. A binomial of the form $\binom{n}{x}$ for $x<0$ will return zero, and the computation will proceed as desired. Hence, closed forms produced using this method are immediately computer-ready. Closed forms for primitive $4$-points $\phi^4$ graphs up to seven loops are included in Appendix~\ref{chartofequations}. All were produced using this method.

\section{The extended graph permanent as an affine hypersurface}\label{affinev}

Some of the first extended graph permanent sequences computed to a reasonable length were similar to those of the $c_2$ invariant (see Section \ref{conclusions} and \cite{BSModForms}). As the $c_2$ invariant is constructed from a point count, it was asked if the extended graph permanent could be expressed as the point count of a polynomial, too. This could potentially open other approaches to understanding the extended graph permanent, and possibly even establish a connection between the $c_2$ invariant and the extended graph permanent.

Let $\mathbb{F}$ be a field. For polynomial $f \in \mathbb{F}[x_1,...,x_n]$, the \emph{affine hypersurface} of $f$ is $$ \left\{(a_1,...,a_n) \in \mathbb{F}^n : f(a_1,...,a_n) = 0 \right\} .$$ In this section, we construct a polynomial from the graph such that, for prime $p$ were the extended graph permanent is defined, the cardinality of the affine hypersurface over $\mathbb{F}_p$ is equal to the extended graph permanent modulo $p$, possibly up to overall sign.

\subsection{A novel graph polynomial}

\begin{definition} Let $F(x_1,...,x_n)$ be a polynomial and $q = p^\alpha$ for some prime $p$. We define the \emph{point count} of $F$ over $q$ to be the number of solutions to $F(x_1,...,x_n) =0$ over $\mathbb{F}_q$, and denote it $[F]_q$. Note that the point count is the cardinality of the affine hypersurface over that field. \end{definition} 

We begin with a previously known method of turning the computation of the permanent into coefficient extraction of a polynomial. For variable $x$, we denote the coefficient of $x$ in function $f$ as $[x]f$. Multivariate coefficient extraction follows as expected.

\begin{definition} Let $A=(a_{ij})$ be an $n \times n$ matrix with integer entries. Define $$F_A(x_1, ... , x_n) = \prod_{i = 1}^n\sum_{j= 1}^n a_{ij}x_j. $$Then, $ \text{Perm}(A) = [x_1 \cdots x_n] F_A.$ This follows immediately from the Leibniz equation for the permanent seen in Definition \ref{permdef}. We will call $F_A$ the \emph{permanent function}. \end{definition}

This function then gives a method of computing the permanent, but given our desire to compute permanents for matrices $M$ and $\mathbf{1}_k \otimes M$, a unique function is needed to compute the permanent of each matrix. Given the block matrix construction, though, we may construct subsequent functions from the permanent function of the fundamental matrix.

\begin{definition} For function $f = a_1 x_1 + \cdots + a_ n x_n$, define the \emph{$r^\text{th}$ extension of $f$} as $$f^{[r]} = a_1x_1 + \cdots + a_nx_n + a_1x_{n+1} + \cdots + a_nx_{2n} + \cdots + a_nx_{rn}.$$ If $F$ is a function that factors into degree one polynomials with no constant terms, $F= f_1 \cdots f_j$, define the \emph{$r^\text{th}$ extension of $F$} as $F^{[r]} = f_1^{[r]} \cdots f_j^{[r]}$. \end{definition}

\begin{remark}\label{extensionmatrix} From the method of computing permanents using a coefficient of the permanent function,  if $$\text{Perm}(A) = [x_1 \cdots x_n] F_A,$$ then $$ \text{Perm}(\mathbf{1}_r \otimes A) = [x_1 \cdots x_{rn}] \left(F_A^{[r]}\right)^r. $$ By construction, $\left( f^{[r]} \right)^r = \left( f^r \right)^{[r]}$.\end{remark}

\begin{proposition}\label{polyextension} Let $h(x_1,...,x_n)$ be a function that factors into degree one polynomials with no constant term. Then, $[x_1 \cdots x_{rn}] h^{[r]} = r!^n [(x_1 \cdots x_n)^r]h$. \end{proposition}

\begin{proof} Let $S_1$ be the permutations of $x_1,...,x_{rn}$ and $S_2$ the permutations of $r$ distinct but indistinguishable copies each of $x_1$, $x_2$, ..., $x_n$. Then, each permutation in $S_2$ appears $r!^n$ times. For $s \in S_t$, $t \in \{1,2\}$, let $s_i$ be the $i^\text{th}$ value in the permutation $s$. Write $h = h_1 \cdots h_k$ as h factored into degree one polynomials, and note that we may assume that $k=rn$ as otherwise the proof is trivial. Then $$[x_1 \cdots x_{rn}]h^{[r]} = \sum_{s \in S_1} \prod_{i=1}^k[s_i]h_i^{[r]}, \hspace{2mm}\text{and } [(x_1\cdots x_n)^r]h = \frac{1}{r!^n} \sum_{s \in S_2} \prod_{i=1}^k[s_i]h_i.$$ These equations follow from the fact that $h$ factors into degree one polynomials.  If $s_i = x_{a+bn}$ for $a,b \in \mathbb{N}$, let $\tilde{s_i} = x_a$. By the construction of these extensions, \begin{align*} 
[x_1 \cdots x_{rn}]h^{[r]} 
&= \sum_{s \in S_1} \prod_{i=1}^k[s_i]h_i^{[r]} \\
&= \sum_{s \in S_1} \prod_{i=1}^k [\tilde{s_i}]h_i \\
&= \sum_{s \in S_2} \prod_{i=1}^k [s_i]h_i = r!^k[(x_1 \cdots x_n)^r]h .\end{align*} This completes the proof. \end{proof}

The Chevalley-Warning Theorem extends to sets of polynomials. Here, we include only the single-polynomial version, as it is sufficient for our needs.

\begin{chev} Let $\mathbb{F}$ be a finite field and $f \in \mathbb{F}[x_1,...,x_n]$ such that $n > \deg(f)$. The number of solutions to $f(x_1,...,x_n) = 0$  is divisible by the characteristic of $\mathbb{F}$. \end{chev}

A proof of this theorem can be found in \cite{Ax}. The following theorem is a corollary to the proof of the Chevalley-Warning Theorem, and will be useful for us to find the appropriate polynomial for our graphs.

\begin{theorem}\label{chevwarn} Let $F$ be a polynomial of degree $N$ in $N$ variables with integer coefficients. Then, $$[(x_1 \cdots x_N)^{p-1}] F^{p-1}  \equiv [F]_p \pmod{p} $$ for primes $p$. \end{theorem}

For our purposes, consider a fundamental matrix for a graph $G$, and let $v' \in V(G)$ be the special vertex. Write $\text{lcm}(|E(G)|, |V(G)|-1) = L$. Let $\V = \frac{L}{|V(G)|-1}$, so the fundamental matrix $\overline{M}$ is a $\V$-matrix. To emphasize the graphic construction, write the permanent function $F_{\overline{M}}$ as $F_{G,v'}$. Then, the permanent function has degree $\frac{L}{\V} \cdot \V = L$ in $L$ variables. 

In order to be able to apply Theorem \ref{chevwarn} we must correct the exponents. Specifically, suppose that $\overline{M}$ is a fundamental matrix and we want to compute the permanent of $\mathbf{1}_r \otimes \overline{M}$ modulo prime $p = r \V+1$. By construction, each factor of $F_{G,v'}$ comes with exponent $\V$. Create polynomial $\widetilde{F}_{G,v'}$ from $F_{G,v'}$ by taking the $\V^\text{th}$ root of $F_{G,v'}$ and then substituting $y_i^\V$ for all $x_i$. As with $F_{G,v'}$, $\widetilde{F}_{G,v'}$ has degree $L$ in $L$ variables.

\begin{example} For the graph $K_4$, we have a fundamental signed incidence matrix
 \begin{align*} \overline{M} = \left( \begin{array}{cccccc} 
1&1&1&0&0&0 \\
-1&0&0&1&1&0 \\
0&-1&0&-1&0&1 \\
1&1&1&0&0&0 \\
-1&0&0&1&1&0 \\
0&-1&0&-1&0&1   \end{array} \right).  \end{align*} This matrix gives permanent functions  \begin{align*}F_{K_4,v'}&= (x_1 + x_2+x_3)^2 (-x_1+x_4+x_5)^2 (-x_2-x_4+x_6)^2, \\  \widetilde{F}_{K_4,v'}&= (y_1^2 + y_2^2+y_3^2) (-y_1^2+y_4^2+y_5^2) (-y_2^2-y_4^2+y_6^2). \end{align*}  \end{example}

\begin{lemma}\label{polyswitch} With variables as defined prior and graph $G$, $[(x_1 \cdots x_L)^r] F_{G,v'}^r = [(y_1 \cdots y_L)^{p-1}] \left( \widetilde{F}_{G,v'} \right)^{p-1}.$ \end{lemma}

\begin{proof} Quickly, \begin{align*} [(x_1 \cdots x_{L})^r] F_{G,v'}^r 
&=  [(x_1 \cdots x_{L})^r] \sqrt[\V]{F_{G,v'}}^{r\V} \\
&=[(y_1^\V \cdots y_{L}^\V)^{r}] (\widetilde{F}_{G,v'})^{r\V}\\
&= [(y_1 \cdots y_{L})^{p-1}] (\widetilde{F}_{G,v'})^{p-1}. \end{align*} \end{proof}

\begin{theorem} Let $G$ be a graph, $L= \text{lcm}(|E(G)|,|V(G)|-1)$, and $\V =\frac{L}{|V(G)|-1}$. Let $F_{G,v'}$ be a permanent function for $G$ with special vertex $v' \in V(G)$. Let $p$ be a prime such that $p \equiv 1 \pmod{\V}$, say $p = r\V+1$. Then $$\text{GPerm}^{[p]} \left(G \right) \equiv r!^L [\widetilde{F}_{G,v'}]_p \pmod{p}.$$ \end{theorem}

\begin{proof} 
\begin{align*}
\text{GPerm}^{[p]}(G) &= \text{Perm}(\mathbf{1}_r \otimes \overline{M})&& \\
&= [x_1 \cdots x_{rL}](F_{G,v'}^{[r]})^r && \text{Remark \ref{extensionmatrix}}\\
&= r!^L [(x_1 \cdots x_L)^r] F_{G,v'}^r && \text{Proposition \ref{polyextension}}\\
&= r!^L [(y_1 \cdots y_L)^{p-1}] (\widetilde{F}_{G,v'})^{p-1} && \text{Lemma \ref{polyswitch}}\\
&\equiv r!^L [\widetilde{F}_{G,v'}]_p \pmod{p} && \text{Theorem \ref{chevwarn}}
\end{align*}
    \end{proof}

Recall Corollary \ref{wilsoncor}, which will be of use to simplify the previous equation for $\phi^4$ graphs; for odd prime $p= 2n+1$, $$n!^2 \equiv \begin{cases} -1\pmod{p} \text{ if $n$ is even} \\ 1 \pmod{p} \text{ if $n$ is odd} \end{cases}.$$

\begin{corollary}\label{maincor} Let $G$ be a $4$-point $\phi^4$ graph, and all variables as defined prior. Then, $$ \text{GPerm}^{[p]}(G) \equiv \begin{cases} [\widetilde{F}_{G,v'}]_p \pmod{p} \text{ if } |E(G)| \equiv 0 \pmod{4} \\ -[\widetilde{F}_{G,v'}]_p \pmod{p} \text{ otherwise} \end{cases} .$$ \end{corollary}

\begin{proof} For a  $4$-point $\phi^4$ graph $G$, $|E(G)|$ is even. By Corollary \ref{wilsoncor} the proof is immediate. \end{proof}

Interestingly, while there is no natural way to include the prime $2$ in the extended graph permanent for $4$-point $\phi^4$ graphs using the permanent construction, it can be extracted from the point count of this polynomial. Since each variable comes with a power of two in this construction, though, and over $\mathbb{F}_2$, we may remove these exponents without loss. For a $\phi^4$ graph $G$, then, $[\widetilde{F}_{G,v'}]_2 \equiv [\sqrt{F_{G,v'}}]_2 \pmod{2}$. Then, $|E(G)| > \deg (\sqrt{F_G})$, and by the Chevalley-Warning Theorem, $[\widetilde{F}_{G,v'}]_2 \equiv 0 \pmod{2}$ for all $4$-point $\phi^4$ graphs.

\subsection{Modular form coefficients}\label{modformscoeffs}

Some extended graph permanent sequences produced were recognizable as Fourier coefficients to a particular type of function; modular forms. We include here a very brief introduction to modular forms. Notational conventions are adapted from \cite{ModForms}. 

The \emph{modular group} is $$\text{SL}_2(\mathbb{Z}) = \left\{ \left[ \begin{array}{cc} a&b \\ c&d  \end{array} \right] : a,b,c,d \in \mathbb{Z}, ad-bc=1 \right\}.$$  Let $\mathcal{H} = \{\tau \in \mathbb{C} : \text{Im}(\tau) > 0\}$, the upper half plane. For $m= \left[ \begin{array}{cc} a&b \\ c&d  \end{array} \right] \in \text{SL}_2(\mathbb{Z})$ and $\tau \in \mathcal{H}$, define fractional linear transformation $$m(\tau) = \frac{a \tau +b}{c\tau + d}, \hspace{2mm} m(\infty) = \frac{a}{c}.$$ Important congruence subgroups for our purposes are \begin{align*} \Gamma(N) &= \left\{ \left[ \begin{array}{cc} a&b \\ c&d  \end{array} \right] \in \text{SL}_2(\mathbb{Z}) : a \equiv d \equiv 1, b \equiv c \equiv 0 \pmod{N} \right\}, \\  \Gamma_0(N) &= \left\{ \left[ \begin{array}{cc} a&b \\ c&d  \end{array} \right] \in \text{SL}_2(\mathbb{Z}) : c \equiv 0 \pmod{N} \right\}, \\ \Gamma_1(N) &= \left\{ \left[ \begin{array}{cc} a&b \\ c&d  \end{array} \right] \in \text{SL}_2(\mathbb{Z}) : a \equiv d \equiv 1, c \equiv 0 \pmod{N} \right\}. \end{align*} A subgroup $\Gamma$ of $\text{SL}_2(\mathbb{Z})$ is a congruence subgroup of \emph{level $N$} if $\Gamma(N) \subseteq \Gamma$ for some $N \in \mathbb{Z}_{>0}$.

\begin{definition} For integer $k$, a function $f:\mathcal{H} \rightarrow \mathbb{C}$ is a \emph{modular form of weight $k$ and level $N$} if $f$ is holomorphic on $\mathcal{H}$ and at infinity, and there is a $k$ in $\mathbb{Z}_{\geq 0}$ such that $$f\left(\left[ \begin{array}{cc} a&b \\ c&d  \end{array} \right](\tau)\right) = (c\tau+d)^k f(\tau)$$ for all $\left[ \begin{array}{cc} a&b \\ c&d  \end{array} \right] \in \Gamma$ and $\tau \in \mathcal{H}$, where $\Gamma$ is one of $\{ \Gamma(N), \Gamma_0(N), \Gamma_1(N)\}$.  \end{definition}

\noindent Our results will use the congruence subgroup $\Gamma_1(N)$ exclusively.

We are interested in sequences generated from the Fourier expansions, known as $q$-expansions, of these modular forms, where $q = e^{2\pi i z}$. Specifically, let $\mathcal{P}$ be the increasing sequence of all primes. For modular form $f$, build a sequence $\left( ([q^p] f) \pmod{p} \right)_{p \in \mathcal{P}}$.

Modular forms are objects of great mathematical interest (see \cite{ModForms}). We were motivated to look for them here by the appearance of modular forms in $c_2$ sequences, another graph invariant conjectured to be preserved by the three graph operations seen in Section \ref{graphoperations} (see Section \ref{conclusions}, as well as \cite{BSModForms} and \cite{Logan}). Here, we find that the sequences from modular forms occasionally appear to match the extended graph permanents. These apparently matching sequences have been checked up to prime $p=97$, and are listed in Table \ref{rockingtable}. The extended graph permanent sequences can be found in Appendix~\ref{chartofgraphs}, and the Fourier expansions can be found at \cite{lmfdb}. The modular forms are listed by their weights and levels. Those that are representible as a Dedekind $\eta$-function product, a modular form of weight $1/2$ commonly written $$ \eta(\tau) = (e^{2\pi i \tau})^\frac{1}{24} \prod_{n=1}^\infty (1-(e^{2\pi i \tau})^n), $$ have this product included. These are taken from \cite{lmfdb}. The graph $(P_{3,1})^2$ is the unique merging of two copies of $P_{3,1}$ per Theorem \ref{2vertexcut}.

\begin{table}
\begin{align*}
\begin{array}{cccc}
\text{Graph}&\text{Weight}&\text{Level } (\Gamma_1) & \text{Modular form} \\ \hline
P_{3,1} & 3  & 16 & -\eta(4z)^6 \\
P_{4,1} & 4 & 8 & \eta(2z)^4\eta(4z)^4 \\
(P_{3,1})^2 & 5 & 4 &  \eta(z)^4\eta(2z)^2\eta(4z)^4\\
P_{6,1}, P_{6,4} & 6 & 4 & \eta(2z)^{12}\\
P_{6,3} & 6 & 8 & \text{not an $\eta$-product}
\end{array} 
\end{align*}
\caption{Modular forms that appear as extended $\phi^4$ graph permanents. The notion of products here refers to the two-vertex joins of graphs established in Theorem \ref{2vertexcut}.}\label{rockingtable}
\end{table}

Some interesting observations can be made here. First, the loop number of the graph is equal to the weight of the modular form in all cases. Secondly, each graph has a modular form with level a power of two. A third observation requires some new terminology. A cusp form is a modular form that has a Fourier expansion with constant term equal to zero. Cusp forms of level $M$ can be embedded into cusp forms of level $N$ for any $N$ that is a multiple of $M$ (see Section 5.6 in \cite{ModForms}). A \emph{newform} is a cusp form that is not directly constructed in this manner. The modular forms in Table~\ref{rockingtable} are all newforms, though that is mainly due to the fact that they were found by searching \cite{lmfdb}. It is interesting, though, that in the Dirichlet character decomposition, a particular  decomposition of this space of newforms, these all fall into subspaces of dimension $1$ (see \cite{lmfdb}).

Aside from the $c_2$ invariant (see \cite{BSModForms} and \cite{Logan}), I am aware of no research regarding $q$-expansions of modular forms as a sequence of residues.

\section{Sign ambiguity} \label{signambiguity}

It is an unfortunate aspect of the arbitrary nature of the underlying edge orientation that a sign ambiguity must exist in the extended graph permanent. Given, however, that we may demand all duplicated edges or copies of the fundamental matrix to preserve this initial orientation, there is only a small loss to the range of possible sequences produced by the invariant. If each edge is duplicated an even number of times the value is not influenced by the edge orientation. In $4$-point $\phi^4$ theory, this corresponds to primes in the sequence of the form $4k+1$ for integer $k$. All other values change sign together with a change in orientation, corresponding to all matrices having an odd number of columns multiplied by $-1$. 

As a result, this does little to reduce the surprise of finding familiar sequences, as in Section \ref{modformscoeffs}. Over the first twelve odd primes, there are approximately $1.52\times 10^{14}$ possible sequences of residues. The extended graph permanent then allows approximately $7.6\times  10^{13}$ sequences. The occurrence of sequences identical to those of the $c_2$ invariant, or modular forms such that the weight of the modular form is equal to the loop number of the graph, is unlikely to be merely a coincidence.

\section{Conclusion} \label{conclusions}

As a potential method of furthering our understanding of the graph period, there is value for any non-trivial graph invariant that is preserved by the Schnetz twist, completion followed by decompletion, and planar duality for $4$-point $\phi^4$ graphs. This motivated the creation of the extended graph permanent. Of course, it follows that we would like to further understand any potential connections between the period and the extended graph permanent.

Currently, two other graph invariants are believed to be preserved by these operations; the $c_2$ invariant and the Hepp bound. For an arbitrary graph $G$ with $|V(G)|>2$, the $c_2$ invariant is defined over the increasing sequence of all primes $\mathcal{P}$ as $$\left( \frac{[\Psi]_p}{p^2} \pmod{p} \right)_{p \in \mathcal{P}},$$ where $\Psi$ is the Kirchhoff polynomial, seen in Equation \ref{kirchhoff}. Equality under duality for the $c_2$ invariant is established in \cite{Dor} and \cite{DorDual}, while invariance under the other operations remains open. For various important structural reasons, the $c_2$ invariant produces sequences common to numerous graphs, including infinite families of graphs with equal sequences. Those with weight drop have sequence $0$ for all primes (see \cite{BrS}), for example. It is interesting then that the extended graph permanent sequences in Appendix~\ref{chartofgraphs} rarely appear to be equal in ways not explainable by the graph operations in Section~\ref{graphoperations}.

The Hepp bound is an upper bound of the period, created by replacing the Kirchhoff polynomial in the period formula with the maximal-weight tree at all points of integration; for a graph with $n$ edges, the Hepp bound is $$\int_{x_2 \geq 0} \cdots \int_{x_n \geq 0} \frac{1}{\left( \max_T \prod_{e \notin T} x_e \right)^2|_{x_1=1}} \prod_{j=2}^{n}\mathrm{d}x_j.$$  It is actually conjectured that two graphs have equal Hepp bound if and only if the two graphs have equal periods (\cite{Schhepp}).

As mentioned prior, the appearance of modular forms was recognized first when an extended graph permanent sequence appeared to be equal to a $c_2$ invariant sequence, albeit for different graphs; the graph $P_{3,1}^2$ appears to have extended graph permanent equal to the $c_2$ invariant of graphs $P_{9,161}$, $P_{9,170}$, $P_{9,183}$, and $P_{9,185}$ (see \cite{BSModForms}). Further, graphs with $c_2$ equal to $-1$ for all primes are of particular interest (\cite{BrS}), and that sequence can be found as the extended graph permanent for some trees. An interesting question then is for which graphs this sequence appears as the extended graph permanent. Trivially, trees are the only connected graphs that have an extended graph permanent value at prime $p=2$, and so describing the sequence by what primes appear forces that trees are the only graphs that match this sequence. If we consider the sequence without regarding the sequence of primes used in the construction, though, the banana graph, mentioned in Section \ref{treesnshit}, also produces this sequence. Apart from these, it is not known if any other graphs do. An interesting question, then, is when will the $c_2$ invariant of one graph be equal to the extended graph permanent of another? Further, is there a graph operation that translates one to the other?

The list of $\phi^4$ graphs up to loop order eight and the extended graph permanents up to prime $p=41$ can be seen in Appendix \ref{chartofgraphs}. Recall Conjecture \ref{obviousconjecture}, that if two graphs have equal periods then they have equal extended graph permanents. While the converse of Conjecture \ref{obviousconjecture} does not appear to hold -- for example, $P_{6,1}$ and $P_{6,4}$ appear to have the same extended graph permanent, but the periods are not equal -- both pair $P_{8,30}$ and $P_{8,36}$ and pair $P_{8,31}$ and $P_{8,35}$ have at least one graph with unknown period, though both pairs are conjectured to have equal periods (\cite{Schhepp}). When two graphs will have equal extended graph permanents but non-equal periods is an interesting problem, and one that requires further study.

In Section \ref{egpcomp}, graphic representations led to closed forms for the permanents of the matrices themselves, as well as the residues. This closed form is a computational boon, as otherwise permanent computations from the matrices themselves can be oppressively difficult. Further, this graphical interpretation of the permanent was useful in establishing a class of matrices with identically zero permanents. While this is a restrictive class, there is a natural extension from graphs to weighted hypergraphs. There is potentially some computational value in a graphical representation of the permanent for the graph theoretic tools it may allow. 

Lastly, the representation of the extended graph permanent as a point count opens up numerous other methods of mathematical approach, and leaves a number of open questions. Immediately, one may ask if there are additional graphs that give modular form sequences, and if those found are indeed equal to the modular forms. Then, if the patterns spotted in Table \ref{rockingtable}, that the weights correspond to the loop number and that levels are powers of two, always hold. Some insight could point to where to look for other equal sequences. Additionally, the fact that the functions for $4$-point $\phi^4$ graphs have point counts over $\mathbb{F}_2$ that vanish modulo two is an interesting aspect. One might ask if other graphs produce functions where the point count also vanishes over finite fields for which the extended graph permanent is not defined.

\section*{Acknowledgments}

I would like to thank Karen Yeats, Matt DeVos, Erik Panzer, and Francis Brown for their helpful notes, ideas, and support.

\definecolor{Gray}{gray}{0.9}
\newcolumntype{g}{>{\columncolor{Gray}}c}

\newpage

\appendix
\section{The EGPs of small $\phi^4$ graphs}
\label{chartofgraphs}

The following charts include the first few primes for all $\phi^4$ graphs up to loop order $8$, as decompletion families of $4$-regular graphs. Graphs with equal sequences resulting from Schnetz twists or duality are noted when applicable. The naming convention comes from \cite{Sphi4}, and representations of the completed graphs can be found there. Graphs with alternate common names are noted, and when it is the decompleted graph that has a common name this will be written in parenthesis. Grey columns mark values that may be thought of as fixed, while all others are defined collectively up to sign for that graph.

Decompletions of graphs $P_{3,1}$, $P_{7,5}$, $P_{7,9}$, $P_{8,18}$, $P_{8,25}$, $P_{8,31}$, and $P_{8,35}$ have zeros in the sequence at all primes congruent to $3$ modulo $4$. In every case, this can be explained by Corollary~\ref{symmetrycor}, as each of these graphs has a decompletion with a symmetry meeting the conditions of this corollary.

There are a number of sets of graphs that appear to have equal extended graph permanents for currently unexplained reasons: $P_{6,1}$ and $P_{6,4}$; $P_{8,1}$, $P_{8,10}$, and $P_{8,40}$; $P_{8,3}$ and $P_{8,32}$; $P_{8,6}$ and $P_{8,39}$; $P_{8,30}$ and $P_{8,36}$; and $P_{31}$ and $P_{8,35}$. Equality of these sequences has been further verified to prime $p=101$. Recall that it is conjectured that the periods of $P_{8,31}$ and $P_{8,35}$ are equal,  and that the periods of $P_{8,30}$ and $P_{8,36}$ are equal. It is perhaps interesting then that when we look at the $c_2$ invariants for the remaining sets of graphs, each graph with seemingly equal extended graph permanent has different $c_2$ invariant.

For purposes of comparison, the first non-zero variate value is chosen to be minimal.

\begin{align*}
\begin{array}{c|cgccggccgcgg}
\textbf{Graph}&\multicolumn{12}{l}{\textbf{Prime}} \\
&3&5&7&11&13&17&19&23&29&31&37&41 \\ \hline
P_{1,1} &1&4&1&1&12&16&1&1&28&1&36&40 \\ \hline
P_{3,1} = C^5_{1,2} = (W_3)  &0&1&0&0&3&13&0&0&16&0&33&23 \\ \hline
P_{4,1} = C^6_{1,2} = (W_4) &1&3&4&0&9&16&13&10&24&5&23&7 \\ \hline
P_{5,1} = C^7_{1,2} &1&1&1&5&12&16&11&13&7&1&25&9 \\ \hline
P_{6,1} = C^8_{1,2} &0&4&3&1&11&16&0&13&15&9&35&6\\
P_{6,2} &1&3&5&8&8&15&10&17&27&20&32&1\\
P_{6,3} &1&1&3&8&10&9&15&0&24&24&3&11\\
P_{6,4} = C^8_{1,3} = (K_{3,4}) &0&4&3&1&11&16&0&13&15&9&35&6 \\ \hline
P_{7,1} = C^9_{1,2} &1&3&3&4&1&15&7&14&13&13&28&0 \\
P_{7,2} &1&2&0&9&9&6&6&12&25&9&0&31 \\
P_{7,3} &0&0&3&8&5&3&2&14&10&18&23&34 \\
P_{7,4} \xleftrightarrow{\text{twist}} P_{7,7} &1&0&4&5&9&1&4&4&4&7&26&0 \\
P_{7,5} \xleftrightarrow{\text{dual}} P_{7,10} &0&3&0&0&1&11&0&0&13&0&26&36\\
P_{7,6}  &1&1&1&8&10&9&7&14&28&16&35&36 
\end{array} 
\end{align*}

\begin{align*}
\begin{array}{c|cgccggccgcgg}
\textbf{Graph}&\multicolumn{12}{l}{\textbf{Prime}} \\
&3&5&7&11&13&17&19&23&29&31&37&41 \\ \hline
P_{7,8} &1&1&2&0&10&16&17&8&4&25&26&33 \\
P_{7,9} &0&0&0&0&10&2&0&0&17&0&1&0 \\
P_{7,11} = C^9_{1,3} &0&1&1&1&11&5&0&22&6&25&16&38 \\ \hline
P_{8,1} = C_{1,2}^{10} &1&1&5&10&7&14&17&4&8&11&19&7\\
P_{8,2} &1&0&4&0&10&6&12&12&27&17&34&0\\
P_{8,3} &1&0&1&1&9&10&14&3&8&17&15&22\\
P_{8,4} &1&3&4&0&7&16&3&11&23&23&11&17\\
P_{8,5} &0&2&1&0&0&16&17&9&12&2&33&26\\
P_{8,6} \xleftrightarrow{\text{twist}} P_{8,9} &0&0&3&0&4&5&6&6&3&13&28&24\\
P_{8,7} \xleftrightarrow{\text{twist}} P_{8,8} &1&1&0&2&0&3&13&2&22&7&25&31\\
P_{8,10}\xleftrightarrow{\text{twist}} P_{8,22} &1&1&5&10&7&14&17&4&8&11&19&7\\
P_{8,11} \xleftrightarrow{\text{twist}} P_{8,15} &1&3&1&1&8&14&0&1&13&20&15&24\\
P_{8,12} &1&1&6&0&7&0&6&15&10&29&11&30 \\
P_{8,13} \xleftrightarrow{\text{twist}} P_{8,21} &1&4&4&7&1&12&7&11&28&11&24&26\\
P_{8,14} &0&3&3&2&2&11&12&3&1&27&30&27\\
P_{8,16} &1&3&1&10&3&1&5&16&3&12&23&5\\
P_{8,17}\xleftrightarrow{\text{twist}} P_{8,23} &0&4&2&0&4&0&9&1&27&7&22&17\\
P_{8,18}\xleftrightarrow{\text{twist}} P_{8,25} &0&3&0&0&0&4&0&0&3&0&15&12\\
P_{8,19} \xleftrightarrow{\text{dual}} P_{8,27} &1&4&4&4&10&2&15&6&3&27&28&36\\
P_{8,20} &1&2&3&2&1&15&6&7&14&25&12&38\\
P_{8,24} &1&2&1&6&7&5&3&5&8&5&25&31\\
P_{8,26}\xleftrightarrow{\text{twist}} P_{8,28} &1&1&0&7&1&10&15&16&6&9&2&12\\
P_{8,29} &1&3&5&8&1&15&13&17&8&23&6&15\\
P_{8,30} &1&4&3&4&6&5&2&21&11&5&34&28\\
P_{8,31} &0&3&0&0&3&1&0&0&25&0&35&13\\
P_{8,32}\xleftrightarrow{\text{twist}} P_{8,34} &1&0&1&1&9&10&14&3&8&17&15&22\\
P_{8,33} &0&1&0&0&7&3&7&19&20&29&3&33\\
P_{8,35} &0&3&0&0&3&1&0&0&25&0&35&13\\
P_{8,36} &1&4&3&4 &6 &5   &2   &21  &11&5 &34 &28\\
P_{8,37} &1&1&5&0 &11&5  &13 &7   &13&30   &16 &15\\
P_{8,38} &1&2&0&1 &1  &4  &6   &15 &11&18 &28 &29\\
P_{8,39} &0&0&3&0&4&5&6&6&3&13&28&24\\
P_{8,40} = C^{10}_{1,4} &1&1&5&10&7 &14 &17 &4  &8  &11 &19 &7\\
P_{8,41} = C^{10}_{1,3} &0&3&1&5&12&2&18&15&9&25&27&34 
\end{array} 
\end{align*}

\newpage

\section{Equations for small $\phi^4$ graphs}
\label{chartofequations}

What follows are equations for the extended graph permanents for primitive $4$-point $\phi^4$ graphs $G$ up to seven loops, using the methods developed in Section~\ref{egpcomp}. Naming conventions come from \cite{Sphi4} as a family of decompletions of a $4$-regular graph. Adopting a shorthand, the summation is from $0$ to $n$ for each variable, though in many instances further restrictions are possible and will speed up computations. The extended graph permanent at prime $p = 2n+1$ is then the residue modulo $p$ for each of these equations, up to a factor of $(-1)^n$ corresponding to changing the direction of an edge in the underlying orientation.

As noted prior, the graphs $P_{7,4}$ and $P_{7,7}$ differ by a Schnetz twist, and graph $P_{7,5}$ and $P_{7,10}$ have decompletions that are planar duals and hence equal sequences of residues by Corollary~\ref{phi4dual}. Both pairs of graphs are included here for comparative purposes.

\scriptsize

\begin{align*}
\begin{array}{l|l}
\textbf{G}&\textbf{Equation} \\ \hline
P_{ 1 , 1 } &(2n)!\\
P_{ 3 , 1 } &
(2n)!^ 3
\sum
\binom{n}{ x }^3
(-1)^{ x }\\
P_{ 4 , 1 } &
(2n)!^ 4
\sum
\binom{n}{ x }^4\\
P_{ 5 , 1 } &
(2n)!^ 5
\sum
\binom{n}{ x_0 }^3
\binom{n}{ x_1 }^2
\binom{n}{ x_0 + x_1 }
(-1)^{x_1 }\\
P_{ 6 , 1 } &
(2n)!^ 6
\sum
\binom{n}{ x_0 }^3
\binom{n}{ x_1 }
\binom{n}{ x_2 }
\binom{n}{ 2n - x_1 - x_2 }
\binom{n}{ x_0 + x_1 }
\binom{n}{ -n + x_0 + x_1 + x_2 }
(-1)^{  x_0 +  x_2 }\\
P_{ 6 , 2 } &
(2n)!^ 6
\sum
\binom{n}{ x_0 }^2
\binom{n}{ x_1 }
\binom{n}{ x_2 }
\binom{n}{  x_1 + x_2 }
\binom{n}{ x_0 + x_1 }^2
\binom{n}{ n - x_0 + x_2 }
(-1)^{ x_0  + x_2 }\\
P_{ 6 , 3 } &
(2n)!^ 6
\sum
\binom{n}{ x_0 }^2
\binom{n}{ x_1 }^2
\binom{n}{ 2n - x_0 - x_1 }
\binom{n}{ x_2 }^2
\binom{n}{ -n + x_0 + x_1 + x_2 }
(-1)^{ x_2 }\\
P_{ 6 , 4 } &
(2n)!^ 6
\sum
\binom{n}{ x_0 }^2
\binom{n}{ x_1 }^2
\binom{n}{ x_2 }^2
\binom{n}{ 2n - x_0 - x_1 - x_2 }^2\\
P_{ 7 , 1 } &
(2n)!^ 7
\sum
\binom{n}{ x_0 }^3
\binom{n}{ x_1 }
\binom{n}{ x_2 }
\binom{n}{ x_3 }
\binom{n}{ 2n - x_1 - x_2 - x_3 }
\binom{n}{ x_0 + x_1 }
\binom{n}{ x_0 + x_1 + x_2 }
\binom{n}{ -x_0 + x_3 }
(-1)^{ x_2 + x_3 }\\
P_{ 7 , 2 } &
(2n)!^ 7
\sum
\binom{n}{ x_0 }^3
\binom{n}{ x_1 }
\binom{n}{ 2n - x_0 - x_1 }
\binom{n}{ x_2 }
\binom{n}{ x_3 }
\binom{n}{ n - x_2 - x_3 }
\binom{n}{ x_1 + x_2 }
\binom{n}{ -n + x_0 + x_1 + x_2 + x_3 }
(-1)^{x_0 + x_3 }\\
P_{ 7 , 3 } &
(2n)!^ 7
\sum
\binom{n}{ x_0 }^2
\binom{n}{ x_1 }^2
\binom{n}{ 2n - x_0 - x_1 }
\binom{n}{ x_2 }
\binom{n}{ x_3 }
\binom{n}{ n - x_2 - x_3 }
\binom{n}{ x_0 + x_2 }
\binom{n}{ -n + x_0 + x_1 + x_2 + x_3 }
(-1)^{x_0 +  x_3 }\\
P_{ 7 , 4 } &
(2n)!^ 7
\sum
\binom{n}{ x_0 }
\binom{n}{ x_1 }^2
\binom{n}{ 2n - x_0 - x_1 }
\binom{n}{ x_2 }
\binom{n}{ x_3 }
\binom{n}{ n - x_2 - x_3 }
\binom{n}{ x_0 + x_2 }^2\\
& \hspace{5cm} \cdot
\binom{n}{ -n + x_0 + x_1 + x_2 + x_3 }
(-1)^{ x_0 +x_2 + x_3 }\\
P_{ 7 , 5 } &
(2n)!^ 7
\sum
\binom{n}{ x_0 }^2
\binom{n}{ x_1 }^2
\binom{n}{ x_2 }
\binom{n}{ x_3 }
\binom{n}{ n - x_2 - x_3 }
\binom{n}{ x_0 + x_2 }^2
\binom{n}{ -x_0 + x_1 + x_3 }
(-1)^{  x_0 + x_1 + x_3 }\\
P_{ 7 , 6 } &
(2n)!^ 7
\sum
\binom{n}{ x_0 }^2
\binom{n}{ x_1 }
\binom{n}{ x_2 }^2
\binom{n}{ x_3 }
\binom{n}{ 2n - x_1 - x_2 - x_3 }
\binom{n}{ x_0 + x_1 }^2
\binom{n}{ -x_0 + x_3 }
(-1)^{ x_0 +  x_2 + x_3 }\\
P_{ 7 , 7 } &
(2n)!^ 7
\sum
\binom{n}{ x_0 }^2
\binom{n}{ x_1 }^2
\binom{n}{ 2n - x_0 - x_1 }
\binom{n}{ x_2 }
\binom{n}{ x_3 }
\binom{n}{ 2n - x_2 - x_3 }
\binom{n}{ -n + x_0 + x_2 }
\binom{n}{ -n + x_1 + x_3 }
(-1)^{x_2 + x_3 }\\
P_{ 7 , 8 } &
(2n)!^ 7
\sum
\binom{n}{ x_0 }^2
\binom{n}{ x_1 }^2
\binom{n}{ x_2 }
\binom{n}{ x_3 }
\binom{n}{ n - x_2 - x_3 }
\binom{n}{ x_0 + x_2 }
\binom{n}{ x_1 + x_3 }
\binom{n}{  x_0 + x_1 + x_2 + x_3 }\\
P_{ 7 , 9 } &
(2n)!^ 7
\sum
\binom{n}{ x_0 }^2
\binom{n}{ x_1 }^2
\binom{n}{ x_2 }
\binom{n}{ x_3 }
\binom{n}{ n - x_2 - x_3 }
\binom{n}{ x_0 + x_2 }
\binom{n}{ x_1 + x_3 }
\binom{n}{ n + x_0 - x_1 - x_3 }
(-1)^{ x_2 }\\
P_{ 7 , 1 0 } &
(2n)!^ 7\sum\binom{n}{ x_0 }^2\binom{n}{ x_1 }\binom{n}{ x_2 }\binom{n}{ x_3 }\binom{n}{ 2n - x_1 - x_2 - x_3 }\binom{n}{ x_0 + x_1 }\binom{n}{ -n + x_0 + x_1 + x_3 }^2 \\  & \hspace{5cm} \cdot \binom{n}{ 2n - x_0 - x_1 - x_2 - x_3 }(-1)^{x_2 + x_3 }  \\
P_{7,11} & (2n)!^7 \sum
\binom{n}{ x_0 }^2
\binom{n}{ x_1 }^2
\binom{n}{ x_2 }
\binom{n}{ x_3 }
\binom{n}{ 2n - x_2 - x_3 }^2
\binom{n}{ -n + x_0 + x_1 + x_2 }
\binom{n}{ -x_0 + x_3 }(-1)^{ x_1  }
\end{array} 
\end{align*}

\normalsize

\newpage

\section{Completed graphs}

We include here the closed form equations and sequences of some of the graphs seen prior, but in the completed forms for comparative purposes. Here, we abuse notation slightly and the name of the graph represents the completed graph, not the family of decompletions.

Both $P_{1,1}$ and $P_{3,1}$ have no sign invariance. The graph $P_{4,1}$ is variate at primes of the form $13 \pmod{24}$. The fixed primes for this graph are marked in grey in the chart.

\noindent $\mathbf{P_{1,1}}$: \begin{align*}\frac{(3n)!^2 (2n)! (-1)^n}{(n!)^2} \pmod{3n+1} \end{align*}

\begin{align*}
\begin{array}{c|cccccccccccc}
\text{Prime}&7&13&19&31&37&43&61&67&73&79&97&103 \\ \hline
\text{GPerm}^{[p]}(P_{1,1}) &6&5&12&27&11&8&60&5&66&17&78&90
\end{array}
\end{align*}

 \noindent $\mathbf{P_{3,1}}$: \begin{align*} &(5n)!^4 \sum \binom{2n}{x_1} \binom{2n}{x_2} \binom{2n}{3n-x_1-x_2} \binom{2n}{x_3} \binom{2n}{n+x_1-x_3} \\ & \hspace{3cm} \cdot \binom{2n}{x_2+x_3-n} (-1)^{n+x_1+x_2+x_3} \pmod{5n+1} \end{align*}

\begin{align*}
\begin{array}{c|cccccccccc}
\text{Prime}&11&31&41&61&71&101&131&151&181&191 \\ \hline
\text{GPerm}^{[p]}(P_{3,1}) &1&6&3&4&41&32&79&8&119&6
\end{array}
\end{align*}

\noindent $\mathbf{P_{4,1}}$: \begin{align*} &(12n)!^5 \sum \binom{5n}{x_1} \binom{5n}{x_2} \binom{5n}{7n-x_1-x_2} \binom{5n}{x_3} \binom{5n}{x_4} \binom{5n}{7n-x_3-x_4} \\ & \hspace{1cm} \cdot \binom{5n}{x_1+x_3-3n} \binom{5n}{x_2+x_4-3n} (-1)^{x_1+x_2+x_3+x_4} \pmod{12n+1} \end{align*}

\begin{align*}
\begin{array}{c|cccggcccgcg}
\text{Prime}&13&37&61&73&97&109&157&181&193&229&241 \\ \hline
\text{GPerm}^{[p]}(P_{4,1}) &2&9&36&41&18&17&5&158&20&114&52
\end{array}
\end{align*}

\newpage

\bibliography{cites}{}

\begin{thebibliography}{10}

\bibitem{Ax}
James Ax.
\newblock Zeroes of polynomials over finite fields.
\newblock {\em Amer. J. Math.}, 86(2):855--261, 1964.

\bibitem{BlEsKr}
Spencer Bloch, H\'{e}l\`{e}ne Esnault, and Dirk Kreimer.
\newblock On motives associated to graph polynomials.
\newblock {\em Comm. Math. Phys.}, (267):181--225, 2006.

\bibitem{BrS}
Francis Brown and Oliver Schnetz.
\newblock A {K3} in $\phi^4$.
\newblock {\em Duke Math. J.}, 161(10):1817--1862, 2012.

\bibitem{BSModForms}
Francis Brown and Oliver Schnetz.
\newblock Modular forms in quantum field theory.
\newblock {\em Commun. Number Theory Phys.}, 7(2):293--325, 2013.

\bibitem{BSZigzag}
Francis Brown and Oliver Schnetz.
\newblock Single-valued multiple polylogarithms and a proof of the zig-zag
  conjecture.
\newblock {\em J. Number Theory}, (148):478--506, 2015.

\bibitem{crump}
Iain Crump, Matt DeVos, and Karen Yeats.
\newblock Period preserving properties of an invariant from the permanent of
  signed incidence matrices.
\newblock {\em Ann. Inst. Henri Poincar\'{e} D}, 3(4):429--454, 2016.

\bibitem{ModForms}
Fred Diamond and Jerry Shurman.
\newblock {\em A First Course in Modular Forms}.
\newblock Springer New York, 2005.

\bibitem{Dor}
Dmitry Doryn.
\newblock The $c_2$ invariant is invariant.
\newblock arXiv:1312.7271.

\bibitem{DorDual}
Dmitry Doryn.
\newblock Dual graph polynomials and a $4$-face formula.
\newblock arXiv: 1508.03484.

\bibitem{lmfdb}
The {LMFDB Collaboration}.
\newblock The {L}-functions and modular forms database.
\newblock http://www.lmfdb.org, 2013.
\newblock [Online; accessed 1 August 2016].

\bibitem{Logan}
Adam Logan.
\newblock New realizations of modular forms in {C}alabi-{Y}au threefolds
  arising from $\phi^4$ theory.
\newblock arXiv:1604.04918.

\bibitem{Oxl}
James Oxley.
\newblock {\em Matroid Theory}.
\newblock Oxford University Press, 2011.

\bibitem{Schhepp}
Oliver Schnetz.
\newblock Numbers and functions in quantum field theory.
\newblock arXiv:1606.08598.

\bibitem{Sphi4}
Oliver Schnetz.
\newblock Quantum periods: A census of $\phi^4$-transcendentals.
\newblock {\em Commun. Number Theory Phys.}, 4(1):1--48, 2010.

\end{thebibliography}
\bibliographystyle{plain}

\end{document}